\documentclass[11pt]{amsart}
\usepackage{amsmath}
\usepackage{amsfonts}
\usepackage{amssymb,amsbsy,amsthm}
\usepackage{graphicx}
\usepackage{color}
\usepackage{amsthm,enumitem}
\numberwithin{equation}{section}
\definecolor{mycolorred}{rgb}{1, 0, 0}

\newtheorem{theorem}{Theorem}[section]

\newtheorem{definition}[theorem]{Definition}

\newtheorem{lemma}[theorem]{Lemma}

\newtheorem{proposition}[theorem]{Proposition}
\newtheorem{remark}[theorem]{Remark}

\def\<{\langle}
\def\>{\rangle}

\def\P{{\mathbb P}}

\def\E{{\mathbb E}}
\def\R{{\mathbb R}}

\def\N{{\mathbb N}}

\def\cL{{\mathcal L}}
\def\cP{{\mathcal P}}

\def\bx{\bar{x}}
\def\bv{\bar{v}}

\def\bX{\bar{X}}
\def\bob{{\mathbf b}}
\def\boc{{\mathbf c}}
\def\bogamma{\boldsymbol{\gamma}}
\def\botheta{\boldsymbol{\theta}}
\def\boxx{{\mathbf x}}
\def\bov{{\mathbf v}}

\def\boX{{\mathbf X}}

\title[Construction of Boltzmann and McKean Vlasov type flows]{Construction of Boltzmann and McKean Vlasov type flows 
  (the sewing lemma approach)}
\date{\today}
\author{Aur\'elien Alfonsi}
\address{CERMICS, Ecole des Ponts, Marne-la-Vall\'ee, France. MathRisk, Inria, Paris,
  France.}
\email{aurelien.alfonsi@enpc.fr}
\author{Vlad Bally}
\address{Universit\'e Paris-Est, LAMA (UMR CNRS, UPEMLV, UPEC), MathRisk INRIA,
  F-77454 Marne-la-Vall\'ee, France.}
\email{bally@univ-mlv.fr}
\thanks{AA acknowledges the support of the ``chaire Risques financiers'', Fondation du Risque.}

\begin{document}

\begin{abstract}
  We are concerned with a mixture of Boltzmann and McKean-Vlasov type equations, this means (in probabilistic
terms) equations with coefficients depending on the law of the solution itself, and driven by a Poisson point measure with the intensity
depending also on the law of the solution. Both the analytical Boltzmann
equation and the probabilistic interpretation initiated by Tanaka~\cite{[T1],[T2]} have
intensively been discussed in the literature for specific models related to
the behavior of gas molecules. In this paper, we  consider general
abstract coefficients that may include mean field effects and then we discuss the link with specific models as
well. In contrast with the usual approach in which integral equations are
used in order to state the problem, we employ here a new formulation of the
problem in terms of flows of endomorphisms on the space of probability
measure endowed with the Wasserstein distance. This point of view already
appeared in the framework of rough differential equations. Our results
concern existence and uniqueness of the solution, in the formulation of
flows, but we also prove that the "flow solution" is a solution of the
classical integral weak equation and admits a probabilistic interpretation. Moreover, we obtain
stability results and regularity with respect to the time for such
solutions. Finally we prove the convergence of empirical measures based on
particle systems to the solution of our problem, and we obtain the rate of
convergence. We discuss as examples the homogeneous and the inhomogeneous
Boltzmann (Enskog) equation with hard potentials.
\end{abstract}

\maketitle

\section{Introduction}

In this paper we consider  a mixture of Boltzmann and McKean-Vlasov type equations defined as follows. Let $\mathcal{P}_{1}({\mathbb{R}}^{d})$ denote the space
of probability measures on ${\mathbb{R}}^{d}$ with a finite first
moment. We consider $\rho \in \mathcal{P}_{1}({\mathbb{R}}^{d})$, an abstract measurable space $(E,\mu )$ and three coefficients $b:{\mathbb{R}}^{d}\times \mathcal{P}_{1}({\mathbb{R}}^{d})\rightarrow {\mathbb{R}}^{d}$, $c:{\mathbb{R}}^{d}\times E\times {\mathbb{R}}^{d}\times \mathcal{P}_{1}({\mathbb{R}}^{d})\rightarrow {\mathbb{R}}^{d}$ and $\gamma :{\mathbb{R}}^{d}\times E\times {\mathbb{R}}^{d}\times \mathcal{P}_{1}({\mathbb{R}}^{d})\rightarrow {\mathbb{R}}_{+}$ that verify some linear growth and some Lipschitz continuity hypothesis (see Assumption $\mathbf{(A)}$ for
precise statements) and we associate the following weak equation on $f_{s,t}\in\cP_1(\R^d)$, $0\le s\le t$: 
\begin{align}
\forall \varphi &\in C_{b}^{1}({\mathbb{R}}^{d}), \  \int_{{\mathbb{R}}^{d}}\varphi (x)f_{s,t}(dx)=\int_{{\mathbb{R}}%
^{d}}\varphi (x)\rho (dx)+\int_{s}^{t}\int_{{\mathbb{R}}^{d}}\left\langle
b(x,f_{s,r}),\nabla \varphi (x)\right\rangle f_{s,r}(dx)dr  \label{int1} \\
&+\int_{s}^{t}\int_{{\mathbb{R}}^{d}\times {\mathbb{R}}%
^{d}}f_{s,r}(dx)f_{s,r}(dv)\int_{E}\left( \varphi
(x+c(v,z,x,f_{s,r}))-\varphi (x)\right) \gamma (v,z,x,f_{s,r})\mu (dz)dr. 
\notag
\end{align}%
Here,  $C_{b}^{1}({\mathbb{R}}^{d})$ denotes the set
of bounded $C^{1}$ functions with bounded gradient. Given a fixed $s\geq 0,$
a solution of this equation is a family $f_{s,t}(dx)\in \mathcal{P}_{1}({\mathbb{R}}^{d}),t\geq s,$ which verify (\ref{int1}) for every test function 
$\varphi$. If one replaced in the last term of~\eqref{int1} the solution $f_{s,r}(dv)$ by a fixed $g_{s,t}(dv)\in \mathcal{P}_{1}({\mathbb{R}}^{d})$, this would be a McKean-Vlasov
type equation. Moreover, when the coefficients $b,c,\gamma $ do not depend on
the solution $f_{s,r}$ and for specific choices of $b$, $c$ and $\gamma$, this
equation covers variants of the Boltzmann equation (see Villani \cite{[V]}
and Alexandre~\cite{[A]} for the mathematical approach and Cercignani~\cite{[C]} for a presentation of the physical background). In the case of the
homogeneous Boltzmann equation the particles are (and remain) uniformly
distributed in space, so their positions do not appear as variables in the equation.
Then, $x\in {\mathbb{R}}^{d}$ represents the velocity of the typical
particle. In this case, the drift coefficient is simply $b=0$. In the case of the
inhomogeneous Boltzmann equation also known as the Enskog equation (see Arkeryd~\cite{[Ar]}), the positions of the particles matter.  One works then on ${\mathbb{R}}^{2d}$, $(x^{1},...,x^{d})$ is the position and $(x^{d+1},...,x^{2d})$ represents the velocity of the typical particle. Then, the drift coefficient will be $b^{i}(x)=x^{i+d},i=1,...,d$ and $b^{i}(x)=0,i=d+1,...,2d$. This is one motivation for considering a general drift term in our abstract formulation.

The probabilistic approach to this type of Boltzmann equation has been
initiated by Tanaka in \cite{[T1]},\cite{[T2]} followed by many others (see 
\cite{[BF]},\cite{[DGM]},\cite{[FG1]},\cite{[FMi]} for example). One takes $f_{s,t}(dx),t\geq s$ to be the solution of the equation (\ref{int1}) and
constructs a Poisson point measure $N_{f}$ with state space ${\mathbb{R}}%
^{d}\times E\times {\mathbb{R}}_{+}$ and with intensity measure 
$f_{s,r}(dv)\mu (dz)1_{{\mathbb{R}}_{+}}(u)du1_{(s,\infty )}(r)dr$.
Then, one associates the stochastic equation%
\begin{align}
  X_{s,t}=X&+\int_{s}^{t}b(X_{s,r},f_{s,r})dr \label{int3}\\
  &+\int_{s}^{t}\int_{{\mathbb{R}}%
^{d}\times E\times {\mathbb{R}}_{+}}c(v,z,X_{s,r-},f_{s,r-})1_{\{u\leq
\gamma (v,z,X_{s,r-},f_{s,r-})\}}N_{f}(dv,dz,du,dr). \notag
\end{align}%
Here, the initial value $X$ is a random variable with law $\rho$ which is independent of the
Poisson measure $N_{f}$. Under suitable hypothesis (for specific
coefficients) one proves that the stochastic equation (\ref{int3}) has a
unique solution and moreover, the law of $X_{s,t}$ is $f_{s,t}(dx).$ In this
sense, (\ref{int3}) is a probabilistic interpretation of (\ref{int1}) and $(X_{s,t})$ is called the "Boltzmann process" (see \cite{[F1]} for example).

In the present paper, we give an alternative formulation of the problem
presented above. We first recall the definition of the Wasserstein distance~$W_{1}$ on the space~$\mathcal{P}_1(\R^d)$: 
\begin{align*}
\mu,\nu \in \mathcal{P}_1(\R^d), \  W_{1}(\mu ,\nu )&=\inf_{\pi \in \Pi (\mu ,\nu )}\int_{{\mathbb{R}}^{d}\times {%
    \mathbb{R}}^{d}}\left\vert x-y\right\vert \pi (dx,dy)\\
&=\sup_{L(f)\leq
1}\left\vert \int_{{\mathbb{R}}^{d}}f(x)\mu (dx)-\int_{{\mathbb{R}}%
^{d}}f(x)\nu (dx)\right\vert,
\end{align*}%
where $\Pi (\mu ,\nu )$ is the set of probability measures on ${\mathbb{R}}%
^{d}\times {\mathbb{R}}^{d}$ with marginals $\mu $ and $\nu$, and $L(f):=\sup_{x\not = y} \frac{|f(y)-f(x)|}{|y-x|}$ is the Lipschitz constant of $f$. The second equality is a classical consequence of Kantorovich duality, see e.g. Remark 6.5~\cite{Villani}. We also introduce $\mathcal{E}_0(\mathcal{P}_{1}({\mathbb{R}}^{d}))$, the metric space of the endomorphisms $\theta :\mathcal{P}_{1}({\mathbb{R}}^{d})\rightarrow \mathcal{P}_{1}({\mathbb{R}}^{d})$ such that $\sup_{\rho \in \mathcal{P}_{1}({\mathbb{%
R}}^{d})}\frac{ \int_{\R^d}
\left\vert x\right\vert \theta(\rho) (dx)}{1+\int_{\R^d}
\left\vert x\right\vert \rho (dx)}< \infty$, endowed with the distance 
\begin{equation*}
d_{\ast }(\theta ,\theta ^{\prime })=\sup_{\rho \in \mathcal{P}_{1}({\mathbb{%
R}}^{d})}\frac{W_{1}(\theta (\rho ),\theta ^{\prime }(\rho ))}{1+\int_{\R^d}
\left\vert x\right\vert \rho (dx)}.
\end{equation*}
This is a complete metric space (see Lemma~\ref{endo_complete}).

Then we construct $\Theta _{s,t}\in \mathcal{E}_0(\mathcal{P}_{1}({\mathbb{R}}^{d}))$ in the following way. Given $\rho \in \mathcal{P}_{1}({\mathbb{R}}^{d})$, we construct a Poisson point measure $N_{\rho }$ with state space ${%
\mathbb{R}}^{d}\times E\times {\mathbb{R}}_{+}$ and with intensity measure 
\begin{equation*}
\widehat{N}_{\rho }(dv,dz,du,dr)=\rho (dv)\mu (dz)1_{{\mathbb{R}}%
_{+}}(u)du1_{\{s,\infty )}(r)dr.
\end{equation*}%
Moreover, we take a random variable $X$ with law $\rho $ which is
independent of the Poisson measure $N_{\rho }$ and we define%
\begin{equation}
X_{s,t}(\rho )=X+b(X,\rho )(t-s)+\int_{s}^{t}\int_{{\mathbb{R}}^{d}\times
E\times {\mathbb{R}}_{+}}c(v,z,X,\rho )1_{\{u\leq \gamma (v,z,X,\rho
)\}}N_{\rho }(dv,dz,du,dr).  \label{int4}
\end{equation}%
Clearly $X_{s,t}(\rho )$ is the one step Euler scheme for the stochastic
equation (\ref{int3}). We define $\Theta _{s,t}(\rho )$ to be the
probability distribution of $X_{s,t}(\rho ):$ 
\begin{equation*}
\Theta _{s,t}(\rho )(dv):={\mathbb{P}}(X_{s,t}(\rho )\in dv).
\end{equation*}%
Under suitable assumptions, we get that $\Theta _{s,t}$ indeed belongs to~$\mathcal{E}_0(\mathcal{P}_{1}({\mathbb{R}}^{d}))$.

\begin{definition}
\label{def_flow} A family of endomorphisms $\theta_{s,t}\in \mathcal{E}_0( \mathcal{P}_{1}({\mathbb{R}}^{d}))$ with $0\leq s<t$ is a flow if 
\begin{equation*}
\theta _{s,t}=\theta _{r,t}\circ \theta _{s,r}, \text{ for every } 0\leq
s<r<t.
\end{equation*}
It is a stationary flow if $\theta_{s,t}=\theta_{0,t-s}$.
\end{definition}
Our problem is stated as follows: find a flow of endomorphisms such that 
\begin{equation}
d_{\ast }(\theta _{s,t},\Theta _{s,t})\leq C(t-s)^{2}.  \label{int5}
\end{equation}%
We call this $\theta$ a "flow solution" of the equation associated to the
coefficients $b,c,\gamma$ and to the measure~$\mu$. It turns out that
under suitable hypotheses, a flow solution exists and is unique. Moreover the
flow solution is a weak solution of Equation~(\ref{int1}) and admits the
stochastic representation~(\ref{int3}).

This special way to characterize the solution of an equation by means of
the distance in short time to the one step Euler scheme first appears,  to our knowledge, in the paper~\cite{[D]} of Davie in the framework of
rough differential equations. Then, Bailleul in~\cite{[B]} and~\cite{[BC]}
coupled this idea with the concept of flows. These ideas appeared in the
framework of rough path integration initiated by Lyons in his seminal paper 
\cite{[L]} (we refer to Friz and Victoir~\cite{[FV]} and to Friz and Hairer~\cite{[FH]} for a complete and friendly presentation of this topic).
It is worth to mention that a central instrument in the rough path theory is
the so called "sewing lemma" introduced by Feyel and De la Pradelle in \cite%
{[FP]},\cite{[FPM]} and in the same time, independently, by Gubinelli in 
\cite{[G]}. \ This is a generic and efficient way to treat the convergence
of Euler type schemes. In our paper we give a general abstract variant of
this lemma which plays a crucial part in our approach. In our framework this
lemma states as follows. We consider an abstract family of endomorphisms $\Theta _{s,t}$ which has the
following two properties. First, we assume the Lipschitz continuity property%
\begin{equation}
d_{\ast }(\Theta _{s,t}\circ U ,\Theta _{s,t}\circ \tilde{U} )\leq
C e^{C(t-s)}d_{\ast }(U ,\tilde{U} ),\quad \forall U ,\tilde{U} \in \mathcal{E}_0(\cP
_{1}({\mathbb{R}}^{d})).  \label{int6}
\end{equation}%
Moreover, notice that $\Theta _{s,t}$ has not the flow property, although we
expect this property to be true for $\theta _{s,t}.$ However, we assume that it has almost this
property in the following  asymptotic (small time) sense: we assume that for every $s<u<t$,
\begin{equation}
d_{\ast }(\Theta_{s,t},\Theta _{u,t}\circ \Theta _{s,u})\leq C(t-s)^{2}.
\label{int7}
\end{equation}%
This is the "sewing property". These two properties essentially  allows to construct by the sewing lemma the flow $\theta _{s,t}$ which satisfies (\ref{int5}) as the limit in $d_{\ast }$ of the Euler schemes based on $%
\Theta _{s,t}.$ More precisely for a partition $\mathcal{P}%
=\{s=s_{0}<....<s_{n}=t\}$\ one defines the Euler scheme $\Theta _{s,t}^{%
\mathcal{P}}=\Theta _{s_{n-1},s_{n}}\circ ....\circ \Theta _{s_{0},s_{1}}$
and constructs $\theta _{s,t}$ as a limit as $\max_{i=1,\dots,n}s_i-s_{i-1}=:\left\vert \mathcal{P}%
\right\vert \rightarrow 0$ of such Euler schemes. Besides, the following error
estimate holds:%
\begin{equation}
d_{\ast }(\Theta _{s,t}^{\mathcal{P}},\theta _{s,t})\leq C\left\vert 
\mathcal{P}\right\vert (t-s).  \label{int8}
\end{equation}

Section~\ref{abs} presents the abstract framework that allows us to prove Lemma~\ref{Sewing}, a generalized sewing lemma that give the existence and the uniqueness of a flow satisfying~\eqref{int5}. A pleasant feature is that the
uniqueness of the flow is quite easy to obtain, since  it  essentially has to match with the limit of Euler schemes. Thus, the flow provides one notable solution of the weak equation~\eqref{int1}: this is the one which is obtained as the limit of Euler schemes. 
In Section~\ref{jump}, we present the framework of our study (i.e. jump type equations) and our main assumptions. We then use this sewing lemma to prove in Theorem~\ref{flow} that the flow $\theta_{s,t}$ defined as the solution of (\ref{int5}) exists and is unique. We further obtain the estimate (\ref{int8}). Besides, we prove that the flow solution $\theta_{s,t}$ constructed
in Theorem \ref{flow} is a weak solution of equation (\ref{int1}) and admits a probabilistic interpretation (see equation (\ref{int3})).
Then, in Section~\ref{particles} we give a numerical approximation scheme for $\theta _{s,t}(\rho )$ based on a particle system. We obtain in Theorem~\ref{theorem_approx} the convergence of the law of any particle towards the flow solution, and give a rate of convergence that is interesting for practical applications. We also obtain a propagation of chaos result for the Wasserstein distance. Last, Section~\ref{sec_Boltz} applies the general results to the homogeneous Boltzmann equation and the non homogeneous Boltzmann (Enskog) equation. The problem of the uniqueness of the homogeneous Boltzmann equation has been studied in several papers by Fournier~\cite{[F1]}, Desvillettes and Mouhot~\cite{[DM]}, Fournier and Mouhot~\cite{[FMu]}. The study of the Enskog equation is up to our knowledge much more recent: we mention here contributions concerning existence, uniqueness and particle system approximations by Albeverio,  R\"udiger and Sundar~\cite{[ARS]} and Friesen,  R\"udiger and Sundar~\cite{[FRS],[FRS1]}. Here, for technical reasons, we only deal with truncated coefficients. The interesting problem of analysing the convergence of the equation with truncated coefficients towards the general equation is not related to our approach based on the sewing lemma and is thus beyond the scope of this paper. We show that the assumptions of Theorem~\ref{flow} are satisfied, which enables to define the flow $\theta_{s,t}$ that is a weak solution of~\eqref{int1} and admits a probabilistic representation~\eqref{int3}. Interestingly, our approach enables us to study equations that combine interactions of Boltzmann type and mean field interactions of McKean-Vlasov type. To illustrate this, we introduce an alternative equation to the Enskog equation, where we replace the space localization function by a mean-field interaction: collisions are more frequent when the typical particle is in a region with a high density of particles. Such a problem enters as well in our framework, and we thus obtain the same results for the flow given by this equation.

\section{Abstract sewing lemma}\label{abs}

We consider an abstract set $V$, and we denote by $\mathcal{E}(V)$ the space
of the endomorphisms $\varphi :V\rightarrow V$. Here and in the rest of the paper, we use the multiplicative notation for composition, so that%
\begin{equation*}
\varphi \psi(v):=\varphi(\psi(v)).
\end{equation*}
We consider $\mathcal{E}_0(V)\subset \mathcal{E}(V)$ a subgroup of endomorphism  (i.e. $I_d\in \mathcal{E}_0(V)$ and $\varphi,\psi\in \mathcal{E}_0(V)\implies \varphi \psi \in \mathcal{E}_0(V)$). We assume that there is a distance $d_\ast$ on $\mathcal{E}_0(V)$ such that $(\mathcal{E}_0(V),d_\ast)$ is a complete metric space.  We assume besides that
\begin{equation}\label{dist_compatibility}
  \forall U\in \mathcal{E}_0(V), \exists C(U) \in \R_+,\forall \varphi,\psi\in \mathcal{E}_0(V), \ d_\ast(\varphi U, \psi U)\le C(U) d_\ast(\varphi, \psi),
\end{equation}
and moreover that we can pick the constant $C(U)$ uniformly in the following sense:
\begin{equation}\label{dist_compatibility2}
  \forall R>0,  \exists \bar{C}_R \in \R_+,\forall U,\varphi,\psi\in \mathcal{E}_0(V), \  d_\ast(U,Id)\le R \implies d_\ast(\varphi U, \psi U)\le \bar{C}_R d_\ast(\varphi, \psi).
\end{equation}
Thanks to~\eqref{dist_compatibility}, we get that $\varphi_n\to \varphi \implies \varphi_n U \to \varphi U$ for any $U\in \mathcal{E}_0$, and~\eqref{dist_compatibility2} ensures that this convergence is uniform on bounded sets.  

We now consider a time horizon $T>0$ which will be fixed in the following and a family of endomorphisms $\Theta _{s,t} \in \mathcal{E}_0(V)$ for $0\leq s\le t \le T$, such that $\Theta_{s,t}=Id$ for $s=t$ and
\begin{equation}
D^\Theta(T):= \sup_{0\leq s\le t \le T}d_\ast(\Theta_{s,t} ,Id)<\infty.  \tag{$\textbf{H}_{0}$}  \label{bis0}
\end{equation}
For a partition $\mathcal{P}=\{s=s_{0}<...<s_{r}=t\}$ of the interval $[s,t]\subset[0,T]$ we
define the corresponding scheme 
\begin{equation*}
\Theta_{s,t}^{\mathcal{P}}:=\Theta_{s_{r-1},s_{r}}\dots
\Theta_{s_0,s_1} \in \mathcal{E}_0(V).
\end{equation*}%

More generally, for $s<t$ and a partition $\mathcal{P}=\{s_{0}<...<s_{r}\}$
such that $s=s_{i}$ and $t=s_{j}$ with $0\le i<j\le r$, we define 
\begin{equation*}
\Theta_{s,t}^{\mathcal{P}}=\Theta _{s_{j-1},s_{j}}\dots
\Theta_{s_i,s_{i+1}}.
\end{equation*}
 For $s \in (0,T)$, we define
\begin{equation}\label{endo_Euls}
  \mathcal{E}^{\Theta}_{s}=\cup_{r\in[0,s]}\{ \Theta_{r,s}^{\mathcal{P}} :  \mathcal{P}=\{r=r_{0}<...<r_{k}=s\} \text{ a partition of }[r,s] \} \subset \mathcal{E}_0(V).
\end{equation}
We assume:
\begin{itemize}
\item (Lipschitz property) There exists $C_{lip}$ such that for any $0\le  s \le t<T$ and $U,\tilde{U}\in\mathcal{E}_0(V)$, 
\begin{equation}
d_\ast(\Theta_{s,t}^{\mathcal{P}}U,\Theta_{s,t}^{\mathcal{P}}\tilde{U})\le C_{lip}d_\ast(U,\tilde{U}).  \tag{$\textbf{H}_{1}$}  \label{bis1}
\end{equation}
\item (Sewing property) There exists $C_{sew}$ and $\beta >1$ such that  any $0\le s<u<t <T$, $U\in \mathcal{E}^{\Theta}_{s}$ 
\begin{equation}
d_{\ast }(\Theta _{s,t} U,\Theta_{u,t}\Theta _{s,u}U )\leq
C_{sew}(t-s)^{\beta }.  \tag{$\textbf{H}_{2}$}  \label{bis2}
\end{equation}
\end{itemize}

We stress that the constants $C_{lip}$ and $C_{sew}$, for $T$ being fixed, do not depend on  $(s,u,t)$ and on $(U,\tilde{U})$.
A family of endomorphisms $\Theta _{s,t}$ that verifies the
hypotheses $(\mathbf{H}_{0})$, $(\mathbf{H}_{1})$ and $(\mathbf{H}_{2})$ will be called a
"semi-flow". In this general framework the "sewing lemma" can be stated as
follows.

\begin{lemma} (Sewing lemma) 
  \label{Sewing} Suppose that \eqref{bis0}, \eqref{bis1} and \eqref{bis2} hold. Then, there exists $\theta_{s,t}\in \mathcal{E}_0(V)$, $0\leq s\le t\le T$, which is a flow (see Definition~\ref{def_flow}) and satisfies
\begin{align}
 & d_{\ast }(\theta _{s,t},\Theta_{s,t})\leq 2^{\beta }C_{lip}C_{sew}\zeta
  (\beta )(t-s)^{\beta },  \label{bis4}
\end{align}
with $\zeta(\beta)=\sum_{n=1}^\infty \frac{1}{n^\beta}$.
Moreover, it satisfies the Lipschitz property
\begin{equation}
     d_\ast(\theta_{s,t}U  ,\theta_{s,t}\tilde{U} )\leq
  C_{lip}d_\ast(U,\tilde{U}) \text{ for } U,\tilde{U} \in  \mathcal{E}_0(V), \label{bis4'} 
\end{equation}
Besides, we have the approximation estimate
\begin{equation}
   d_{\ast }(\Theta _{s,t}^{\mathcal{P}},\theta _{s,t})\leq 2^{\beta
}C_{lip}^{2}C_{sew}\zeta (\beta )(t-s)\left\vert \mathcal{P}\right\vert
^{\beta -1} \text{ for } \mathcal{P} \text{ partition of } [s,t] ,\label{bis5}
\end{equation}
with $\left\vert \mathcal{P}\right\vert :=\max_{i=0,...,r-1}(s_{i+1}-s_{i}).$

Furthermore, this is the unique flow such that $d_{\ast }(\theta _{s,t},\Theta _{s,t} )\leq C (t-s) h(t-s)$   for some constant $C>0$ and nondecreasing function $h:\R_+\to \R_+$ such that $\lim_{t\to 0}h(t)=0$. 
\end{lemma}

\begin{proof}
  We first prove that for any $U\in \mathcal{E}_s^\Theta$,
\begin{equation}
 d_{\ast }(\Theta_{s,t}^{\mathcal{P}} U ,\Theta_{s,t} U  )\leq
 2^{\beta}C_{lip}C_{sew}\zeta (\beta )(t-s)^{\beta }.  \label{bis6}
 \end{equation}
  We consider $\mathcal{P}=\{s=s_{0}<...<s_{r}=t\}$ and prove~\eqref{bis6} by iteration on~$r$. For $r=1$, the
inequality is obvious. Let $r\ge 2$. For a fixed $i\in\{1,\dots, r-1\}$, we
denote by $\mathcal{P}_{i}$ the partition in which we have canceled $s_{i}$.
Then, we have 
\begin{equation*}
d_{\ast }(\Theta_{s,t}^{\mathcal{P}}U,\Theta_{s,t}^{\mathcal{P}%
_{i}}U)=d_{\ast }(\Theta_{s_{i+1},t}^{\mathcal{P}}ZU,\Theta_{s_{i+1},t}^{%
\mathcal{P}}Z^{\prime }U)
\end{equation*}%
with 
\begin{equation*}
Y=\Theta_{s,s_{i-1}}^{\mathcal{P}}, \quad Z=\Theta _{s_{i},s_{i+1}}\Theta
_{s_{i-1},s_{i}}Y \text{ and } Z^{\prime }=\Theta _{s_{i-1},s_{i+1}}Y.
\end{equation*}%
Using (\ref{bis1}) first and (\ref{bis2}) next we obtain%
\begin{align}
d_{\ast }(\Theta_{s,t}^{\mathcal{P}}U,\Theta_{s,t}^{\mathcal{P}_{i}}U)& \leq
C_{lip}d_{\ast }(ZU,Z^{\prime }U)=C_{lip}d_{\ast}(\Theta
_{s_{i-1},s_{i+1}}YU,\Theta _{s_{i},s_{i+1}}\Theta _{s_{i-1},s_{i}}YU))
\label{bis6'} \\
& \leq C_{lip}C_{sew}(s_{i+1}-s_{i-1})^{\beta }.  \notag
\end{align}%
We give now the sewing argument. We choose $i_{0}\in \{1,\dots, r-1\}$ such
that%
\begin{equation*}
s_{i_{0}+1}-s_{i_{0}-1}\leq \frac{2}{r-1}(t-s).
\end{equation*}%
Such an $i_{0}$ exists, otherwise we would have $2(t-s)\geq
\sum_{i=1}^{r-1}(s_{i+1}-s_{i-1})>2(t-s)).$ Using the inequality (\ref{bis6'}%
) for this $i_{0}$ we obtain%
\begin{equation*}
d_{\ast }(\Theta_{s,t}^{\mathcal{P}}U,\Theta_{s,t}^{\mathcal{P}%
_{i_{0}}}U)\leq \frac{2^{\beta }C_{lip}C_{sew}}{(r-1)^{\beta }}(t-s)^{\beta
}.
\end{equation*}%
We iterate this procedure up to the trivial partition $\{s<t\}$, and we
obtain (\ref{bis6}).

We are now in position to define $\theta_{s,t}$ as the limit of $\Theta_{s,t}^{\mathcal{P}}$
when $\left\vert \mathcal{P}\right\vert $ goes to zero. Since~$(\mathcal{E}_0(V),d_\ast)$ is
complete, it is sufficient to check the Cauchy criterion: 
\begin{equation*}
\lim_{\left\vert \mathcal{P}\right\vert \vee \left\vert \overline{\mathcal{P}%
}\right\vert \rightarrow 0}d_{\ast }(\Theta _{s,t}^{\mathcal{P}},\Theta
_{s,t}^{\overline{\mathcal{P}}})=0.
\end{equation*}%
Let $\mathcal{P}\cup \overline{\mathcal{P}}$ denote the partition of $[s,t]$
obtained by merging both partitions. Since $d_{\ast }(\Theta _{s,t}^{%
\mathcal{P}},\Theta _{s,t}^{\overline{\mathcal{P}}})\leq d_{\ast }(\Theta
_{s,t}^{\mathcal{P}},\Theta _{s,t}^{\mathcal{P}\cup \overline{\mathcal{P}}%
})+d_{\ast }(\Theta _{s,t}^{\overline{\mathcal{P}}},\Theta _{s,t}^{\mathcal{P%
}\cup \overline{\mathcal{P}}})$, we may assume without loss of generality
that $\overline{\mathcal{P}}$ is a refinement of the partition $\mathcal{P}$%
. Thus, we can write $\mathcal{P}=\{s=s_{0}<...<s_{r}=t\}$ and $\overline{%
\mathcal{P}}=\cup _{i=1}^{r}\mathcal{P}^{i}$, where $\mathcal{P}^{i}$ is a
partition of $[s_{i-1},s_{i}]$. We now introduce for $l\in \{0,\dots ,r\}$
the partition $\overline{\mathcal{P}}_{l}$ of $[s,t]$ defined by 
\begin{equation*}
\overline{\mathcal{P}}_{0}=\mathcal{P} \text{ and } \overline{\mathcal{P}}_{l}=\left( \cup _{i=1}^{l}\mathcal{P}^{i}\right) \cup 
\mathcal{P} \text{ for } l\ge 1.
\end{equation*}%
So, $\overline{\mathcal{P}}_{l}$ is the partition in which we refine the
intervals $[s_{i-1},s_{i}]$, $i=1,\dots ,l$, according to $\overline{\mathcal{P}}$ but we do not refine the intervals $[s_{i-1},s_{i}]$, $i=l+1,\dots ,r$ (we keep them unchanged, as they are in $\mathcal{P}$). Thus, we have $\overline{%
\mathcal{P}}_{0}=\mathcal{P}$ and $\overline{\mathcal{P}}_{r}=\overline{%
\mathcal{P}}$, and we obtain by using the triangle inequality: 
\begin{equation*}
d_{\ast }(\Theta _{s,t}^{\mathcal{P}},\Theta _{s,t}^{\overline{\mathcal{P}}%
})\leq \sum_{l=0}^{r-1}d_{\ast }(\Theta _{s,t}^{\overline{\mathcal{P}}%
_{l+1}},\Theta _{s,t}^{\overline{\mathcal{P}}_{l}}).
\end{equation*}%
We note $\varphi _{l}=\Theta _{s_{l+1},t}^{\overline{\mathcal{P}}},\psi
_{l}=\Theta _{s,s_{l}}^{\mathcal{P}}$ and have 
\begin{equation*}
d_{\ast }(\Theta _{s,t}^{\overline{\mathcal{P}}_{l+1}},\Theta _{s,t}^{%
\overline{\mathcal{P}}_{l}})=d_{\ast }(\varphi _{l}\Theta _{s_{l},s_{l+1}}^{%
\mathcal{P}^{l}}\psi _{l},\varphi _{l}\Theta _{s_{l},s_{l+1}}\psi _{l}).
\end{equation*}%
Using first (\ref{bis1}) and then (\ref{bis6}), we obtain 
\begin{equation*}
d_{\ast }(\Theta _{s,t}^{\overline{\mathcal{P}}_{l+1}},\Theta _{s,t}^{%
\overline{\mathcal{P}}_{l}})\leq C_{lip}d_{\ast }(\Theta
_{s_{l},s_{l+1}}^{\mathcal{P}^{l}}\psi _{l},\Theta _{s_{l},s_{l+1}}\psi
_{l})\leq C(s_{l+1}-s_{l-1})^{\beta },
\end{equation*}%
with $C=2^{\beta }C_{lip}^{2}C_{sew}\zeta (\beta )$. This leads to 
\begin{equation*}
d_{\ast }(\Theta _{s,t}^{\mathcal{P}},\Theta _{s,t}^{\overline{\mathcal{P}}%
})\leq C\sum_{l=0}^{r-1}(s_{l+1}-s_{l-1})^{\beta }\le C(t-s)\left\vert \mathcal{%
P}\right\vert ^{\beta -1}\underset{|\mathcal{P}|\rightarrow 0}{\rightarrow }%
0.
\end{equation*}%
This shows the existence of $\theta_{s,t}$ together with~\eqref{bis5}. We then get easily~\eqref{bis4'} :  by sending $|\cP|\to 0$, we get the Lipschitz property of~$\theta_{s,t}$ from~\eqref{bis1}.  Moreover, we get~\eqref{bis4} $d_\ast(\theta_{s,t}U,\Theta_{s,t}U)\le 2^{\beta}C_{lip}C_{sew}\zeta (\beta )(t-s)^{\beta }$ from~\eqref{bis6} when $|\cP|\to 0$, and we simply take $U=Id$.

We now prove the flow property. Let $s,u,t$ be such that $0\leq s<u<t\leq T$, $\cP_1$ and $\cP_2$ be respectively a partition of  $[s,u]$ and  $[u,t]$. We have by using the triangle inequality,~\eqref{bis1} and~\eqref{dist_compatibility}
\begin{align*}
  d_{\ast}(\Theta^{\cP_2}_{u,t} \Theta^{\cP_1}_{s,u} , \theta_{u,t}\theta_{s,u})&\le d_{\ast}(\Theta^{\cP_2}_{u,t} \Theta^{\cP_1}_{s,u} ,  \Theta^{\cP_2}_{u,t} \theta_{s,u})+ d_{\ast}(\Theta^{\cP_2}_{u,t} \theta_{s,u} , \theta_{u,t}\theta_{s,u}) \\
  &\le C_{lip}d_{\ast}( \Theta^{\cP_1}_{s,u} ,  \theta_{s,u})+C(\theta_{s,u})d_{\ast}(\Theta^{\cP_2}_{u,t}  , \theta_{u,t} ) \to 0,
\end{align*}
as  $|\cP_1|\vee  |\cP_2|\to 0$. The concatenation $\cP_1\cup\cP_2$ is a partition of $[s,t]$ and thus $d_{\ast}(\Theta^{\cP_2}_{u,t} \Theta^{\cP_1}_{s,u},\theta_{s,t})\to 0$. We get $d_{\ast}(\theta_{s,t} , \theta_{u,t}\theta_{s,u})=0$ and so $\theta_{s,t} = \theta_{u,t}\theta_{s,u}$. 

We finally prove the uniqueness. Let $\tilde{\theta}_{s,t}\in \mathcal{E}_0(V)$, $0\le s\le t \le T$, be a family satisfying the flow property $\tilde{\theta}_{s,t} = \tilde{\theta}_{u,t}\tilde{\theta}_{s,u}$, $d_{\ast}(\tilde{\theta}_{s,t}U,\Theta_{s,t}U)\le C(t-s)^\beta$ for any $U\in \mathcal{E}^\Theta_s$. We consider a partition $\mathcal{P}=\{s=s_{0}<...<s_{r}=t\}$ and have by using first the flow property and the  triangle inequality, second the Lipschitz property: 
\begin{align*}
  d_\ast(\tilde{\theta}_{s,t},\Theta^{\cP}_{s,t})&\le \sum_{i=0}^{r-1}d_\ast(\Theta^{\cP}_{s_i,t}\tilde{\theta}_{s,s_i},\Theta^{\cP}_{s_{i+1},t}\tilde{\theta}_{s,s_{i+1}})\\&
  \le C\sum_{i=0}^{r-1}d_\ast(\Theta_{s_i,s_{i+1}} \tilde{\theta}_{s,s_i},\tilde{\theta}_{s_i,s_{i+1}}\tilde{\theta}_{s,s_i})
\end{align*}
Now, we observe that $d_\ast(\tilde{\theta}_{s,s_i},I_d)\le C T^\beta + d_\ast(\Theta_{s,s_i},I_d)\le  C T^\beta +D^\Theta(T)=:R(T)$ by using~\eqref{bis0}. Thanks to the uniform bound~\eqref{dist_compatibility2}, we get $d_\ast(\Theta_{s_i,s_{i+1}} \tilde{\theta}_{s,s_i},\tilde{\theta}_{s_i,s_{i+1}}\tilde{\theta}_{s,s_i})\le \bar{C}_{ R(T)} d_\ast(\Theta_{s_i,s_{i+1}}  ,\tilde{\theta}_{s_i,s_{i+1}})$ and thus
$$  d_\ast(\tilde{\theta}_{s,t},\Theta^{\cP}_{s,t})\le C^2 \bar{C}_{R(T)}\sum_{i=0}^{r-1}(s_{i+1}-s_i)h(s_{i+1}-s_i) \le C^2 \bar{C}_{ R(T)}(t-s)h(|\cP|).$$
This yields to $\theta_{s,t}=\tilde{\theta}_{s,t}$ by taking $|\cP|\to 0$.
\end{proof}

\begin{remark}
The hypothesis (\ref{bis2}) may be weakened by replacing $(t-s)^{\beta }$ by $(t-s) (1 \vee \left\vert \ln (t-s)\right\vert ^{-\rho })$ for some $\rho >1$. The
proof is exactly the same by using the fact that the series $\sum_{n} \frac 1{n(\ln n)^{\rho }}$ converges iff $\rho >1$. But in this case, the estimates in (\ref{bis4}) and (\ref{bis5}) are
less explicit. So, we keep (\ref{bis2}) which is verified in our framework of Section~\ref{jump}
with $\beta =2$.
\end{remark}
\begin{remark}\label{Rk_S} We observe from the proof of Lemma~\ref{Sewing} that the uniform bound~\eqref{dist_compatibility2} on the distance and Hypothesis~\eqref{bis0} are only needed for the uniqueness result of Lemma~\ref{Sewing}.
\end{remark}

We now present a typical setting that falls into our framework. 
We consider a complete metric space $(V,d)$ and  $v_0$ be a fixed element of $V$.
We define
\begin{equation}\label{def_E0_metric}\mathcal{E}_0(V)=\{\Theta\in \mathcal{E}(V): \sup_{v \in V}\frac{d(v_0,\Theta(v))}{1+d(v_0,v)} <\infty \},
\end{equation}
the set of endomorphisms with sublinear growth. We endow $\mathcal{E}_0(V)$ with the following distance
\begin{equation}\label{def_d*_metric}
  d_{\ast }(\Theta,\overline{\Theta})=\sup_{v \in V}\frac{d(\Theta(v),\overline{\Theta}(v))}{1+d(v_0,v)}, \  \Theta,\overline{\Theta}\in\mathcal{E}_0(V) .
\end{equation}
It is clear from the definition of~$\mathcal{E}_0(V)$ that $d_{\ast }(\Theta,\overline{\Theta})<\infty$ and the distance properties of $d_\ast$ are obviously inherited from those of~$d$. We remark also that if $\Theta_1\in\mathcal{E}_0(V)$ and $\Theta_2\in \mathcal{E}(V)$ is such that $d_\ast(\Theta_1,\Theta_2)<\infty$, then $\Theta_2\in \mathcal{E}_0(V)$. Besides, we check also easily  that $Id\in\mathcal{E}_0(V)$ and $(\mathcal{E}_0(V),\circ)$ is a group: for $\Theta_1,\Theta_2\in\mathcal{E}_0(V)$,
$$ \sup_{v \in V}\frac{d(v_0,\Theta_2(\Theta_1(v)))}{1+d(v_0,v)}\le  \sup_{v \in V}\frac{d(v_0,\Theta_2(\Theta_1(v)))}{1+d(v_0,\Theta_1(v))} \sup_{v \in V}\frac{1+d(v_0,\Theta_1(v))}{1+d(v_0,v)} <\infty.$$

\begin{lemma}
\label{endo_complete} $(\mathcal{E}_0(V),d_{\ast })$ defined by~\eqref{def_E0_metric} and~\eqref{def_d*_metric} is a complete metric space. Besides, \eqref{dist_compatibility2} holds.
\end{lemma}
\begin{proof}
  Let $\theta_n\in \mathcal{E}_0(V)$ be a sequence such that $\sup_{p,q\ge n}d_{\ast }(\theta_p ,\theta_q)\underset{n\to \infty}\to 0$. Then, for any $v \in  V$, there exists $\theta_\infty(v)\in V$ such that $d(\theta_n (v ),\theta_\infty(v ))\to 0$ since $(V,d)$ is complete.  Therefore, we have
  $d(\theta_n (v ),\theta_\infty(v ))\le \sup_{q\ge n}d(\theta_n (v ),\theta_q(v ))$, which gives $d_{\ast }(\theta_n ,\theta_\infty)\le \sup_{q\ge n} d_{\ast }(\theta_n ,\theta_q) \to 0$.

  We now consider $U,\varphi,\psi \in \mathcal{E}_0(V)$ and have
  \begin{align*}
    d_\ast(\varphi U, \psi U)&= \sup_{v\in V} \frac{d(\varphi (U(v)),\psi(U(v)))}{1+d(v_0,U(v))} \frac{1+d(v_0,U(v))}{1+d(v_0,v)}\\
    &\le C(U)d_\ast(\varphi  , \psi ),
  \end{align*}
  with $C(U)=\sup_{v\in V} \frac{1+d(v_0,U(v))}{1+d(v_0,v)}<\infty$ since $U\in \mathcal{E}_0(V)$. Using that $d(v_0,U(v))\le d(v_0,v)+d(v,U(v))$, we get $C(U)\le 1+d_\ast(U,Id)$ and we therefore obtain \eqref{dist_compatibility2}.
\end{proof}
Last, it is interesting to discuss in this setting the properties~\eqref{bis1} and~\eqref{bis2}. We first have:
\begin{equation}\label{equivH1}
  \eqref{bis1}\iff \forall v_1,v_1\in V, d(\Theta^{\cP}_{s,t}(v_1),\Theta^{\cP}_{s,t}(v_2))\le C_{lip}d(v_1,v_2). 
\end{equation}
To get the direct implication, we take $U(v)=v_1$, $\tilde{U}(v)=v_2$ and observe that $d_\ast(U,\tilde{U})=\sup_{v\in V} \frac{d(v_1,v_2)}{1+d(v_0,v)}=d(v_1,v_2)$. The other implication is clear from~\eqref{def_d*_metric}.
Besides, we have
\begin{align}
  &\exists  \tilde{C}_{sew}:V\to \R_+ ,\ d(\Theta_{s,t}(v),\Theta_{u,t}(\Theta_{s,u}(v))) \le \tilde{C}_{sew}(v)(t-s)^\beta \text{ and } \sup_{v\in V} \frac{\tilde{C}_{sew}(v)}{1+d(v_0,v)}<\infty,  \notag \\
  & \text{ and } \bar{D}:=\sup_{0\leq s \le T} \sup_{U\in \mathcal{E}_s^\Theta} d_\ast(U ,Id)<\infty \implies \eqref{bis2}. \label{implic_H2}
\end{align}
In fact, we then have by~\eqref{dist_compatibility2}
$$ d_\ast(\Theta_{s,t}U,\Theta_{u,t}\Theta_{s,u}U)\le C_{\bar{D}}  \sup_{v\in V} \frac{\tilde{C}_{sew}(v)}{1+d(v_0,v)} (t-s)^\beta,$$
which gives~\eqref{bis2} with $C_{sew}=C_{\bar{D}}  \sup_{v\in V} \frac{\tilde{C}_{sew}(v)}{1+d(v_0,v)}$.

\section{Jump type equations\label{jump}}


\subsection{Framework and assumptions}
We recall that $\mathcal{P}_{1}({\mathbb{R}}^{d})$ is the space of probability
measures $\nu $ on ${\mathbb{R}}^{d}$ such that $\int \left\vert x\right\vert d\nu (x)<\infty$, and  $W_{1}$ is the $1$-Wasserstein
distance on $\mathcal{P}_{1}({\mathbb{R}}^{d})$ defined by
\begin{equation*}
W_{1}(\mu ,\nu )=\inf_{\pi }\int_{{\mathbb{R}}^{d}}\left\vert x-y\right\vert
\pi (dx,dy)
\end{equation*}%
with the infimum taken over all the probability measures $\pi $ on ${\mathbb{R}}^{d}\times {\mathbb{R}}^{d}$ with marginals $\mu $ and~$\nu$. We will
work with endomorphisms $\Theta :\mathcal{P}_{1}({\mathbb{R}}%
^{d})\rightarrow \mathcal{P}_{1}({\mathbb{R}}^{d})$ and we denote by $%
\mathcal{E}(\mathcal{P}_{1}({\mathbb{R}}^{d}))$ the space of these endomorphisms
and
$$\mathcal{E}_0(\mathcal{P}_{1}({\mathbb{R}}^{d}))=\{ \Theta \in \mathcal{E}(\mathcal{P}_{1}({\mathbb{R}}^{d})) :  \sup_{\rho \in \mathcal{P}_{1}({%
\mathbb{R}}^{d})}\frac{ \int \left\vert v\right\vert \Theta(\rho) (dv)}{1+\int \left\vert v\right\vert \rho (dv)}<\infty \}.$$
On
this space we define the distance 
\begin{equation*}
d_{\ast }(\Theta ,\overline{\Theta })=\sup_{\rho \in \mathcal{P}_{1}({%
\mathbb{R}}^{d})}\frac{W_{1}(\Theta (\rho ),\overline{\Theta }(\rho ))}{%
1+\int_{\R^d} \left\vert v\right\vert \rho (dv)}.
\end{equation*}%
We are precisely in the framework presented at the end of Section~\ref{abs} with $V=\cP_1(\R^d)$, $d=W_1$ and $v_0=\delta_0$ is the Dirac mass at~$0$ since $W_1(\rho,\delta_0)=\int_{\R^d} \left\vert v\right\vert \rho (dv)$. It is well known that $(\cP_1(\R^d),W_1)$ is a complete metric space (see e.g. Bolley~\cite{Bolley}) and we get from Lemma~\ref{endo_complete} that $(\mathcal{E(P}_{1}({\mathbb{R}}^{d})),d_{\ast })$ is a complete space that satisfies~\eqref{dist_compatibility2}, so we can apply the results of
Section~\ref{abs}. Finally, for a random variable $X$ we denote by $\mathcal{L}(X)$ the probability law of $X$.

We define now, for $s\leq t$ the semi-flow $\Theta _{s,t}$ in the following way. We consider a
measurable space $(E,\mu )$ and three functions $b:{\mathbb{R}}^{d}\times 
\mathcal{P}_{1}({\mathbb{R}}^{d})\rightarrow {\mathbb{R}}^{d}$, $c:{\mathbb{R}}^{d}\times E\times {\mathbb{R}}%
^{d}\times 
\mathcal{P}_{1}({\mathbb{R}}^{d})\rightarrow {\mathbb{R}}^{d}$ and $\gamma :{\mathbb{R}}^{d}\times
E\times {\mathbb{R}}^{d}\times \mathcal{P}_{1}({\mathbb{R}}^{d})\rightarrow {%
\mathbb{R}}_{+}$ (we give below the precise hypothesis). We denote%
\begin{equation}
Q(v,z,u,x,\nu )=c(v,z,x,\nu )1_{\{u\leq \gamma (v,z,x,\nu )\}}, \text{ for } u\ge 0.  \label{h5}
\end{equation}%
Then, for a probability measure $\rho \in \mathcal{P}_{1}({\mathbb{R}}^{d})$
we take $X\in L^{1}(\Omega )$ with law $\mathcal{L}(X)=\rho $ and we define%
\begin{equation}
X_{s,t}(X)=X+b(X,\rho )(t-s)+\int_{s}^{t}\int_{{\mathbb{R}}^{d}\times
E\times {\mathbb{R}}_{+}}Q(v,z,u,X,\rho )N_{\rho }(dv,dz,du,dr).  \label{W2}
\end{equation}%
Here, $N_{\rho }$ is a Poisson point measure with intensity measure%
\begin{equation}
\widehat{N}_{\rho }(dv,dz,du,dr)=\rho (dv)\mu (dz)dudr.  \label{W3}
\end{equation}%
We stress that the law of $X$ appears in the intensity of the point process.
Moreover, we define 
\begin{equation}
\Theta _{s,t}(\rho )=\mathcal{L}(X_{s,t}(X)).  \label{W3'}
\end{equation}%
So $\Theta _{s,t}(\rho )$ is the law of the solution which has initial value
with distribution $\rho $. Our aim is to construct the flow
corresponding to the semi flow $\Theta _{s,t}$  by using the sewing
lemma~\ref{Sewing}.

Before going on, we precise the hypotheses that we require for the coefficients. We
make the three following assumptions:
\begin{itemize}
\item The drift coefficient $b$ is globally Lipschitz continuous: we assume
that 
\begin{equation}
\exists L_{b}\in {\mathbb{R}}_{+}^{\ast },\ \left\vert b(x,\nu )-b(y,\rho
)\right\vert \leq L_{b}(\left\vert x-y\right\vert +W_{1}(\nu ,\rho )) 
\tag{${\textbf{A}}_{1}$}  \label{lipb}
\end{equation}

\item For every $(v,x)\in {\mathbb{R}}^{d}\times {\mathbb{R}}^{d}$ there
exists a function $Q_{v,x}:{\mathbb{R}}^{d}\times E\times \R_{+}\times {%
\mathbb{R}}^{d}\times \mathcal{P}_{1}({\mathbb{R}}^{d})\rightarrow {\mathbb{R%
}}^{d}$ such that for every $v,x,v^{\prime },x^{\prime }\in {\mathbb{R}}%
^{d},\rho \in \mathcal{P}_{1}({\mathbb{R}}^{d})$ and for every $\varphi \in
C_{b}^{0}({\mathbb{R}}^{d})$%
\begin{equation}
\int_{E\times {\mathbb{R}}_{+}}\varphi (Q(v^{\prime },z,u,x^{\prime },\rho
))\mu (dz)du=\int_{E\times {\mathbb{R}}_{+}}\varphi (Q_{v,x}(v^{\prime
},z,u,x^{\prime },\rho ))\mu (dz)du. 
\tag{${\textbf{A}}_{2}$}  \label{h6e}
\end{equation}%
We assume that $(v,x,v^{\prime },z,u,x,\rho )\rightarrow Q_{v,x}(v^{\prime
},z,u,x^{\prime },\rho )$ is jointly measurable.

\item We assume that $\int |Q(0,z,u,0,\delta _{0})|\mu (dz)du<\infty $
and that there exists a constant $L_{\mu }(c,\gamma )$ such that for every $%
v_{1},x_{1},v_{2},x_{2}\in {\mathbb{R}}^{d}$ and $\rho _{1},\rho _{2}\in 
\mathcal{P}_{1}({\mathbb{R}}^{d})$%
\begin{align}
&\int_{E\times {\mathbb{R}}_{+}}\left\vert Q(v_{1},z,u,x_{1},\rho
_{1})-Q_{v_{1},x_{1}}(v_{2},z,u,x_{2},\rho _{2})\right\vert \mu (dz)du \tag{$\textbf{A}_3$}  \label{h6d} \\
&\leq L_{\mu }(c,\gamma )(\left\vert x_{1}-x_{2}\right\vert +\left\vert
v_{1}-v_{2}\right\vert +W_{1}(\rho _{1},\rho _{2})). \notag
\end{align}
\end{itemize}
We will simply say that $(\mathbf{A})$ is satisfied when these three
Assumptions~\eqref{lipb},~\eqref{h6e}, and~\eqref{h6d} are fulfilled.

\begin{remark}
  The (pseudo) Lipschitz condition~\eqref{h6d} looks at first sight surprising. In some cases such as the two-dimensional Boltzmann equation, one may simply take $Q_{v,x}=Q$ since we have a standard Lipschitz property. For some other interesting models such as the three-dimensional Boltzmann equation, we need to use a non-trivial transformation $Q_{v,x}$ satisfying~\eqref{h6e}. The difficulty comes from the fact that there is no smooth parametrisation of the unit sphere in ${\mathbb{R}}^{3}$, and Tanaka~\cite{[T1]} has been able to get around this difficulty by using such transformation.
\end{remark}

From~\eqref{lipb} and~\eqref{h6d}, we easily deduce the following sublinear
growth estimates for any $x,v\in {\mathbb{R}}^{d}$ and $\rho \in \mathcal{P}_{1}({\mathbb{R}}^{d})$:
\begin{align}
&\left\vert b(x,\rho )\right\vert \leq \left\vert b(0,\delta _{0})\right\vert
+L_{b}(\left\vert x\right\vert +W_{1}(\rho ,\delta _{0}))= \left\vert b(0,\delta _{0})\right\vert
+L_{b}\left(\left\vert x\right\vert +\int_{\R^d}|x|\rho(dx) \right),  \label{growthb}
\\
& \int_{E\times {\mathbb{R}}_{+}}\left\vert Q(v,z,u,x,\rho
)\right\vert \mu (dz)du\leq C_{\mu }(c,\gamma )(1+\left\vert v\right\vert
+\left\vert x\right\vert +W_{1}(\rho ,\delta _{0})),  \label{h6a}
\end{align}
with $C_{\mu }(c,\gamma )=L_{\mu }(c,\gamma )\vee \left( \int
|Q(0,z,u,0,\delta _{0})|\mu (dz)du\right) $. In particular (\ref{growthb})
and (\ref{h6a}) implies%
\begin{equation}
{\mathbb{E}}(\left\vert X_{s,t}(X)-X\right\vert )\leq \left[ |b(0,\delta_0)|+C_{\mu
}(c,\gamma )+(2L_{b}+3C_{\mu }(c,\gamma ))\int \left\vert v\right\vert \rho
(dv)\right] (t-s).  \label{h6b}
\end{equation}%
This ensures that $\Theta _{s,t}(\rho )\in \mathcal{P}_{1}({\mathbb{R}}^{d})$ for $\rho \in \mathcal{P}_{1}({\mathbb{R}}^{d})$ and  that $\Theta_{s,t}\in \mathcal{E}_0(\mathcal{P}_{1}({\mathbb{R}^d}))$.

\subsection{Preliminary results}

We give now a stability result for the one step Euler scheme which is a key ingredient with our approach.

\begin{lemma}\label{STABILITY} Let us assume that the coefficients $b$ and $Q$ satisfy Assumption~$(\mathbf{A})$ with constants $L_b$ and $L_\mu(c,\gamma)$. For $i=1,2$, we consider a $\R^d$-valued random variable $Z^{i}\in L^{1}$
  and a family of probability measures $f_{s,t}^{i}(dv) \in \mathcal{P}_1(\R^d)$ for $0\le s\leq t\leq T$ such that $[s,T] \ni t\mapsto f_{s,t}^{i}(dv)$ is continuous in Wasserstein distance. Let $N_{f^{i}}$ be a Poisson point processes independent of~$Z^i$ with intensity measure $f_{s,t}^{i}(dv)\mu (dz)1_{\R_{+}}(u)dudt$ and  let $(X_{s,t}^{i},t\ge s)$ be defined by
\begin{equation}
X_{s,t}^{i}(Z^{i},\rho^{i})=Z^{i}+\int_{s}^{t}b(Z^i,\rho^{i})dr+\int_{s}^{t}\int_{\R^{d}\times E\times \R_{+}}Q(v,z,u,Z^i,\rho^{i})N_{f^{i}}(dv,dz,du,dr),  \label{NEW1}
\end{equation}%
where $\rho^{i}\in \mathcal{P}_{1}({\mathbb{R}}^{d})$. Then, we have: 
\begin{align}
W_{1}(\mathcal{L}(X_{s,t}^{1}(Z^{1},\rho^{1}))&,\mathcal{L}
(X_{s,t}^{2}(Z^{2},\rho^{2})))  \leq W_{1}(\mathcal{L}(Z^{1}),\mathcal{L}(Z^{2})) +L_{\mu }(c,\gamma)\int_{s}^{t}W_{1}(f_{s,r}^{1},f_{s,r}^{2})dr\notag \\&+(L_{b}+L_{\mu }(c,\gamma)) (W_{1}(\rho^{1},\rho^{2})+W_{1}(\mathcal{L}(Z^{1}),\mathcal{L}(Z^{2})))(t-s).
 \label{NEW2''}
\end{align}
\end{lemma}

\begin{proof}
  We first recall the following useful lemma.
\begin{lemma}
\label{lemSK} There exists a measurable map $\psi :[0,1)\times \mathcal{P}_{1}({\mathbb{R}}^{d})\rightarrow {\mathbb{R}}^{d}$ such that 
\begin{equation*}
\forall f\in \mathcal{P}_{1}({\mathbb{R}}^{d}), \forall \varphi :{\mathbb{R}}^{d}\rightarrow {\mathbb{R}} \text{ bounded measurable},\ \int_{0}^{1}\varphi(\psi (u,f))du=\int_{{\mathbb{R}}^{d}}\varphi (x)f(dx).
\end{equation*}
\end{lemma}
\noindent This result is stated in~\cite{[CD]} (p.~391, Lemma 5.29) in a $L^{2}$ framework, but their proof works the same in our setting.

Let $\pi_{0}^{Z} \in \mathcal{P}_1(\mathbb{R}^{d}\times \mathbb{R}^{d})$ be an optimal coupling for $W_1$ of $\mathcal{L}(Z^{1})$ and $\mathcal{L}(Z^{2})$, that is
\begin{equation*}
W_{1}(\mathcal{L}(Z^{1}),\mathcal{L}(Z^{2}))=\int_{{\mathbb{R}}^{d}\times {\mathbb{R}}^{d}}\left\vert z_{1}-z_{2}\right\vert \pi_{0}^{Z}(dz_{1},dz_{2}).
\end{equation*}%
Then, we construct a random variable $\overline{Z}=(\overline{Z}^{1},\overline{Z}^{2})$ such that $\overline{Z}\sim \pi _{0}^{Z}.$ In particular
we will have 
\begin{equation}
\E(\vert \overline{Z}^{1}-\overline{Z}^{2}\vert )=W_{1}(\mathcal{L}%
(Z^{1}),\mathcal{L}(Z^{2})).  \label{NEW2'}
\end{equation}%
Moreover, for every $s\leq t$ we consider a probability measure $\pi_{s,t}^{f}(dv_{1},dv_{2})$ on $\mathbb{R}^{d}\times \mathbb{R}^{d}$ which is
an optimal $W_1$-coupling between $f_{s,t}^{1}(dv_{1})$ and $f_{s,t}^{2}(dv_{2})$,
and we construct $\tau _{s,t}^{f}(w)=(\tau _{s,t}^{f_{1}}(w),\tau_{s,t}^{f_{2}}(w))$ which represents $\pi_{s,t}^{f}$ in the sense of Lemma~\ref{lemSK}, this means%
\begin{equation*}
\int_{0}^{1}\varphi (\tau _{s,t}^{f}(w))dw=\int_{{\mathbb{R}}^{d}\times {\mathbb{R}}^{d}}\varphi (v_{1},v_{2})\pi _{s,t}^{f}(dv_{1},dv_{2}).
\end{equation*}%
In particular, we have
\begin{equation}\label{Wass_tau}
\int_{0}^{1}\left\vert \tau _{s,t}^{f_{1}}(w)-\tau _{s,t}^{f_{2}}(w)\right\vert
dw=\int_{{\mathbb{R}}^{d}\times {\mathbb{R}}^{d}}\left\vert
v_{1}-v_{2}\right\vert \pi
_{s,t}^{f}(dv_{1},dv_{2})=W_{1}(f_{s,t}^{1},f_{s,t}^{2}).
\end{equation}%
Since $t\mapsto f^i_{s,t}$ is continuous in Wasserstein distance (hence measurable), we may
construct $\tau _{s,t}^{f}(w)$ to be jointly measurable in $(s,t,w)$ by using Corollary 5.22~\cite{Villani}.

Now, we consider $N(dw,dz,du,dr)$ a Poisson point measure on $[0,1]\times E \times {\mathbb{R}}_{+}\times {\mathbb{R}}_{+},$ with intensity measure $dw\mu (dz)dudr$.  We
stress that $\overline{Z}^{1}$ and $\overline{Z}^{2}$ are independent of the
Poisson point measure $N$. We then consider the equation 
\begin{equation*}
x_{s,t}^{i}(\overline{Z}^{i},\rho^{i})=\overline{Z}^{i}+\int_{s}^{t}b(\overline{Z}^{i},\rho^{i}) dr+\int_{s}^{t}\int_{[0,1]\times E\times {\mathbb{R}}_{+}} Q(\tau_{s,r}^{f^{i}}(w),z,u,\overline{Z}^{i},\rho^{i}) N(dw,dz,du,dr)
\end{equation*}%
and we notice that the law of $x_{s,t}^{i}(\overline{Z}^{i},\rho^{i})$
coincides with the law of $X_{s,t}^{i}(\overline{Z}^{i},\rho^{i})$ and so
with the law of $X_{s,t}^{i}(Z^{i},\rho^{i})$. We also note that these laws have a first finite moment thanks to Assumption~$(\textbf{A})$ on $(b,Q)$ and~\eqref{h6b}.
Moreover, we define 
\begin{equation*}
y_{s,t}(\overline{Z}^{1},\rho^{1})=\overline{Z}^{1}+\int_{s}^{t}b(\overline{Z}^{1},\rho^{1})dr+\int_{s}^{t}\int_{[0,1]\times E\times {\mathbb{R}}_{+}}Q_{\tau_{s,r}^{f^{2}}(w),\overline{Z}^{2}}(\tau_{s,r}^{f^{1}}(w),z,u,\overline{Z}^{1},\rho^{1})N(dw,dz,du,dr).
\end{equation*}
By~(\ref{h6e}), the law of $y_{s,t}(\overline{Z}^{1},\rho^{1})$ coincides with the law of $x_{s,t}^{1}(\overline{Z}^{1},\rho^{1})$, so with
the law of $X_{s,t}^{1}(Z^{1},\rho^{1})$. Therefore, we get by using the triangle inequality together with~(\ref{NEW2'}) and Assumption~\eqref{lipb} for the last inequality:
\begin{align*}
W_{1}(\mathcal{L}(X_{s,t}^{1}(Z^{1},\rho^{1})),\mathcal{L}(X_{s,t}^{2}(Z^{2},\rho^{2}))) &=W_{1}(\mathcal{L}(y_{s,t}(\overline{Z}^{1},\rho^{1})),\mathcal{L}(x_{s,t}^{2}(\overline{Z}^{2},\rho^{2}))) \\
&\leq \E\left(\left\vert y_{s,t}(\overline{Z}^{1},\rho^{1})-x_{s,t}^{2}(%
\overline{Z}^{2},\rho^{2})\right\vert \right) \\
&\leq W_{1}(\mathcal{L(}Z^{1}),\mathcal{L}(Z^{2}))+L_{b}\left(W_{1}(\rho^{1},\rho^{2})+\E(\vert \overline{Z}^{1}-\overline{Z}^{2}\vert
)\right)(t-s) \\
& \ \ +\Delta _{s,t},
\end{align*}
with%
\begin{align*}
\Delta_{s,t} &=\E\int_{s}^{t}\int_{[0,1]\times E\times {\mathbb{R}}_{+}}\left\vert Q_{\tau _{s,r}^{f^{2}}(w),\overline{Z}^{2}}(\tau
_{s,r}^{f^{1}}(w),z,u,\overline{Z}^{1},\rho^{1})-Q(\tau_{s,r}^{f^{2}}(w),z,u,\overline{Z}^{2},\rho^{2}) \right\vert dw\mu (dz)dudr.
\end{align*}
Using the pseudo-Lipschitz property (\ref{h6d}) and~\eqref{Wass_tau}, we obtain%
\begin{align*}
  \Delta _{s,t}^{1} &\leq L_{\mu }(c,\gamma )\left(\int_{s}^{t}\left[W_{1}(\rho^{1},\rho^{2})+\E
    (\vert \overline{Z}^{1}-\overline{Z}%
^{2}\vert )+\int_{0}^{1}\left\vert \tau _{r}^{f^{1}}(w)-\tau
_{r}^{f^{2}}(w)\right\vert dw\right]dr \right) \\
&=L_{\mu }(c,\gamma )\left( [W_{1}(\rho^{1},\rho^{2})+W_{1}(\mathcal{L}(Z^{1}),\mathcal{L}(Z^{2}))] (t-s)+\int_{s}^{t} W_{1}(f_{s,r}^{1},f_{s,r}^{2})dr \right),
\end{align*}
which leads to  (\ref{NEW2''}).
\end{proof}

\subsection{Flows of measures}\label{Flow}

We go on and prove the sewing and Lipschitz properties for the semi-flow $%
\Theta_{s,t}$ defined in (\ref{W3'}).

\begin{lemma}\label{lem_sew_lip}
Suppose that $\mathbf{(A)}$ holds. Then, for every $0\le s<u<t\le T$ and for every $%
\rho \in \mathcal{P}_{1}({\mathbb{R}}^{d})$%
\begin{equation}
W_{1}(\Theta _{s,t}(\rho ),\Theta _{u,t}(\Theta _{s,u}(\rho )))\leq \tilde{C}_{sew}(\rho) (t-s)^{2}  \label{W1}
\end{equation}%
with%
\begin{equation}
\tilde{C}_{sew}(\rho)=(L_{b}+2L_{\mu }(c,\gamma ))\left[ |b(0)|+C_{\mu }(c,\gamma
)+(L_{b}+2C_{\mu }(c,\gamma ))\int \left\vert v\right\vert \rho (dv)\right]
\label{W1a}
\end{equation}%
Moreover, for every $\rho ,\xi \in \mathcal{P}_{1}({\mathbb{R}}^{d})$ 
\begin{eqnarray}
W_{1}(\Theta _{s,t}(\rho ),\Theta _{s,t}(\xi )) &\leq &C_{lip}(T) W_{1}(\rho
,\xi )\quad \text{ with }  \label{W1'} \\
C_{lip}(T) &=&1+(2L_{b}+3L_{\mu }(c,\gamma ))T.  \label{W1'a}
\end{eqnarray}
\end{lemma}

\begin{proof} The estimate (\ref{W1'}) is a direct consequence of the estimate~\eqref{NEW2''} obtained in Lemma~\ref{STABILITY}. Let us prove (\ref{W1}). We take $X$ a random variable with law $\rho $ and we consider $X_{s,t}(X)$ defined in (\ref{W2}). So, by
definition, $\Theta _{s,t}(\rho )=\mathcal{L}(X_{s,t}(X)).$ Moreover, we
take $u\in (s,t)$ and we denote $Y=X_{s,u}(X).$ Then we write%
\begin{equation*}
X_{s,t}(X)=Y+b(X,\rho )(t-u)+\int_{u}^{t}\int_{\R^{d}\times E\times
\R_{+}}Q(v,z,u,X,\rho )N_{\rho }(dv,dz,du,dr).
\end{equation*}%
We also denote $\overline{\rho }=\mathcal{L}(Y)=\Theta _{s,u}(\rho )$ and we
define 
\begin{equation*}
Y_{u,t}(Y)=Y+b(Y,\overline{\rho })(t-u)+\int_{u}^{t}\int_{\R^{d}\times
E\times \R_{+}}Q(v,z,u,Y,\overline{\rho })N_{\overline{\rho }}(dv,dz,du,dr).
\end{equation*}%
We notice that the law of $Y_{u,t}(Y)$ is $\Theta_{u,t}(\overline{\rho })=\Theta _{u,t}(\Theta _{s,u}(\rho ))$. Using again (\ref{NEW2''}) on $[u,t]$, it follows that 
\begin{eqnarray*}
W_{1}(\Theta _{s,t}(\rho ),\Theta _{u,t}(\Theta _{s,u}(\rho ))) &=&W_{1}\left(
\mathcal{L}(X_{s,t}(X)),\mathcal{L}(Y_{u,t}(Y))\right) \\
&\leq &(L_{b}+2L_{\mu }(c,\gamma ))(t-u)(W_{1}(\mathcal{L}(X),\mathcal{L}(Y))+W_{1}(\rho ,\overline{\rho }))\\&=&2(L_{b}+2L_{\mu }(c,\gamma ))(t-u)W_{1}(\rho ,\overline{\rho }).
\end{eqnarray*}%
By (\ref{h6b}), we get 
\begin{equation*}
W_{1}(\rho ,\overline{\rho })=W_{1}(\mathcal{L}(X),\mathcal{L}(Y))\leq
\E(\left\vert X-Y\right\vert )=\E(\left\vert X-X_{s,u}(X)\right\vert )\leq
C(u-s)
\end{equation*}%
with $C$ given in (\ref{h6b}). So (\ref{W1}) is proved. 
\end{proof}

\medskip
\noindent We are now able to use the abstract sewing lemma in order to construct the solution of our problem. 
\begin{theorem}
\label{flow}Suppose that $\mathbf{(A)}$ holds true. Then, there exists a
unique flow $\theta_{s,t} \in \mathcal{E}_0(\mathcal{P}_{1}({\mathbb{R}}^{d}))$ with $0\le s< t\le T$ such that%
\begin{equation}
\exists C_1,C>0, \ d_{\ast }(\theta_{s,t},\Theta_{s,t})\leq C_{1}(t-s)^{2}.  \label{W6}
\end{equation}%
Moreover, we have for every partition $\mathcal{P}$ of $[s,t]$ 
\begin{equation}
d_{\ast }(\theta _{s,t},\Theta _{s,t}^{\mathcal{P}})\leq
C_{2}(t-s)\left\vert \mathcal{P}\right\vert,  \label{W7}
\end{equation}%
and~\eqref{W6} and~\eqref{W7} are satisfied with $C_{1}=4\zeta (2 )C_{lip}C_{sew}$ and $C_{2}=C_{lip}C_{1}$ (see (\ref{Csew}) and (\ref{W1'a}) for the values of $C_{sew}$ and of $C_{lip}$).
Besides, $\theta $ is stationary in the sense that $\theta_{s,t}=\theta_{0,t-s}$ and the map $t\mapsto \theta _{s,t}$ is Lipschitz continuous for $d_\ast $, i.e. \begin{equation}\label{lipc0_flow}
  d_\ast(\theta _{s,t},\theta _{s,t^{\prime }})\le C|t^{\prime }-t|.
\end{equation}
Last, we have the following stability result:
\begin{equation}\forall 0\le s\le t \le T, \rho,\xi \in \cP_1(\R^d),\quad  W_1(\theta_{s,t}(\rho),\theta_{s,t}(\xi))\le  \exp\left[ (2L_b+3L_\mu(c,\gamma))(t-s) \right] W_1(\rho,\xi).\label{stab_flot}
\end{equation}
\end{theorem}

\begin{proof}
  We will use the sewing lemma~\ref{Sewing} and check the different assumptions. \eqref{bis0} is a straightforward consequence of~\eqref{h6b}.
  From~\eqref{NEW2''}, we get that $$W_1(\Theta_{s,t}(\rho),\Theta_{s,t}(\xi))\le [1+(L_b+L_\mu(c,\gamma))(t-s)]W_1(\rho,\xi)\le \exp\left( (2L_b+3L_\mu(c,\gamma))(t-s)\right) W_1(\rho,\xi).$$  By iterating, this gives for any partition $\cP$ of $[s,t]$
  \begin{equation}\label{estim_flot}
    W_1(\Theta^\cP_{s,t}(\rho),\Theta^\cP_{s,t}(\xi))\le \exp\left[ (2L_b+3L_\mu(c,\gamma))(t-s)\right] W_1(\rho,\xi).
  \end{equation}
Property~\eqref{bis1} thus holds by using~\eqref{equivH1}.
 
  We now check~\eqref{bis2} and will use the recipe given by~\eqref{implic_H2}. We get from~\eqref{W1}:
  \begin{align*}
&    d_\ast(\Theta_{s,t}\Theta_{r,s}^{\mathcal{P}^{\prime }} , \Theta_{u,t} \Theta_{s,u} \Theta_{r,s}^{\mathcal{P}^{\prime }})\\&\le  (L_b+2L_\mu(c,\gamma))(s-t)^2\sup_{\rho \in \mathcal{P}_1(\R^d)} \frac{|b(0)|+ C_\mu(c,\gamma) + (L_b+2C_\mu(c,\gamma))\int_{\R^d}|v|\Theta_{r,s}^{\mathcal{P}^{\prime }}(\rho)(dv) }{1+\int_{\R^d}|v|\rho(dv)}
  \end{align*}
We now observe that for $\cP'=\{r=r_0<\dots<r_m=s\}$, we have by~\eqref{growthb} and~\eqref{h6a}
$$\int_{\R^d}|v|\Theta_{r,r_{i+1}}^{\mathcal{P}^{\prime }}(\rho)(dv)\le\left( 1+(2L_b+3C_\mu(c,\gamma))(r_{i+1}-r_i) \right) \int_{\R^d}|v|\Theta_{r,r_{i}}^{\mathcal{P}^{\prime }}(\rho)(dv)+ (r_{i+1}-r_i)[|b(0,\delta_0)|+C_\mu(c,\gamma)].$$
   By  iterating this inequality, we get \begin{equation}\label{moment_theta}\int_{\R^d}|v|\Theta_{r,s}^{\mathcal{P}^{\prime }}(\rho)(dv)\le\exp((2L_b+3C_\mu(c,\gamma)) T) \left( \int_{\R^d}|v|\rho(dv) + [|b(0,\delta_0)|+C_\mu(c,\gamma)]T \right).
  \end{equation}
  Therefore the supremum over~$\rho$ is upper bounded by some constant $\overline{C}>0$, and \eqref{bis2} holds with $\beta=2$ and
  \begin{equation}\label{Csew}
    C_{sew}=   (L_b+2L_\mu(c,\gamma))\overline{C} .
  \end{equation}
We can thus apply Lemma~\ref{Sewing} to get the existence and uniqueness of the flow~$\theta$. Then, we easily get~\eqref{stab_flot} from~\eqref{estim_flot} and the stationarity is a simple consequence of the obvious fact that $\Theta_{s,t}=\Theta _{0,t-s}$.

To prove the Lipschitz property~\eqref{lipc0_flow}, we consider $t'\in [s,t]$ and get by iterating the upper bound~\eqref{h6b}
\begin{equation*}
W_1(\Theta _{s,t}^{\mathcal{P}}(\rho ),\Theta _{s,t^{\prime }}^{%
\mathcal{P}}(\rho ))\le \left[ |b(0)|+C_\mu(c,\gamma) + (L_b+2C_\mu(c,\gamma)) \max_{u \in \mathcal{P}} \int_{\R^d} |v| \Theta_{s,u} (\rho)(dv)  \right]|t-t^{\prime }|,
\end{equation*}%
for any partition $\mathcal{P}$ of $[0,T]$ such that $t'\in \mathcal{P}$, since $\Theta _{s,t}^{\mathcal{P}}(\rho )$ can be obtained from $\Theta_{s,t'}^{\mathcal{P}}(\rho )$ by applying the Euler scheme. Using~\eqref{moment_theta}, we get $d_\ast(\Theta _{s,t}^{\mathcal{P}}(\rho ),\Theta _{s,t^{\prime }}^{\mathcal{P}}(\rho ))\le \overline{C}|t'-t|$, and we conclude by sending $|\mathcal{P}|\to 0$.   
\end{proof}

\subsection{The weak equation\label{weak}}

Theorem~\ref{flow} provides a unique solution in terms of flows. Now we
prove that this solution solves an integral equation, in the weak sense. 
Namely, for every $s\geq 0$ and $\rho \in \mathcal{P}({\mathbb{R}}^{d})$\ we
associate the weak equation 
\begin{align}
\int_{{\mathbb{R}}^{d}}\varphi (x)\theta _{s,t}(\rho )(dx) =&\int_{{\mathbb{%
R}}^{d}}\varphi (x)\rho (dx)+\int_{s}^{t}\int_{{\mathbb{R}}^{d}}\left\langle
b(x,\theta _{s,r}(\rho )),\nabla \varphi (x)\right\rangle \theta _{s,r}(\rho
)(dx)dr  \label{we2} \\
&+\int_{s}^{t}\int_{{\mathbb{R}}^{d}\times {\mathbb{R}}^{d}}\Lambda
_{\varphi }(v,x,\theta _{s,r}(\rho ))\theta _{s,r}(\rho )(dx)\theta
_{s,r}(\rho )(dv)dr,\ \varphi \in C_{b}^{1}({\mathbb{R}}^{d}),  \notag
\end{align}
where 
\begin{eqnarray}
\Lambda _{\varphi }(v,x,\rho ) &=&\int_{E\times \R_{+}}(\varphi
(x+Q(v,z,u,x,\rho ))-\varphi (x))\mu (dz)du  \label{w2'} \\
&=&\int_{0}^{1}d\lambda \int_{E\times \R_{+}}\left\langle \nabla \varphi
(x+\lambda Q(v,z,u,x,\rho )),Q(v,z,u,x,\rho )\right\rangle \mu (dz)du
\label{w2''} 
\end{eqnarray}%
Here, we have written equation~\eqref{we2} with a double indexed family of probability measures $(\theta_{s,t}(dx), 0\le s\le t)$. This is to make a direct link with the flow constructed by the sewing lemma that is naturally double indexed. In the literature,  one usually rather considers the following equation for a family of
probability measures $(f_{t}(\rho ),t\geq 0)$
\begin{eqnarray}
\int_{{\mathbb{R}}^{d}}\varphi (x)f_{t}(\rho )(dx) &=&\int_{{\mathbb{R}}%
^{d}}\varphi (x)\rho (dx)+\int_{0}^{t}\int_{{\mathbb{R}}^{d}}\left\langle
b(x,f_{r}(\rho )),\nabla \varphi (x)\right\rangle f_{r}(\rho )(dx)dr
\label{we2bis} \\
&&+\int_{0}^{t}\int_{{\mathbb{R}}^{d}\times {\mathbb{R}}^{d}}\Lambda
_{\varphi }(v,x,f_{r}(\rho ))f_{r}(\rho )(dx)f_{r}(\rho )(dv)dr.  \notag
\end{eqnarray}%
The link between~\eqref{we2} and~\eqref{we2bis} is clear. If $\theta_{s,t}$ solves~\eqref{we2}, then $\theta_{0,t}$ solves~\eqref{we2bis}. Conversely, if $f_{t}$ solves~\eqref{we2bis}, then $f_{t-s}$ solves~\eqref{we2}.

We need the following preliminary lemma.
\begin{lemma}\label{lem_lambdaphi} Assume  that $\mathbf{(A)}$ holds.  If $\varphi \in C^1_b(\R^d)$, then we have
  \begin{align}
    &|\Lambda_\varphi(v,x,\rho)|\le C_\mu(c,\gamma) \|\nabla \varphi\|_\infty\left(1+|v|+|x|+\int_{\R^d} |z|\rho(dz) \right) \label{sub_lin_A4}\\
    &|\Lambda_{\varphi }(v,x,\rho)-\Lambda_{\varphi }(v',x,\rho')|\le L_\mu(c,\gamma) \|\nabla \varphi\|_\infty (|v-v'|+W_1(\rho,\rho')), \label{liplambda}
  \end{align}
   and 
  \begin{equation}\label{hyp_cont}
    \forall v \in \R^d,\rho \in \cP_1(\R^d), \ x\mapsto \Lambda_\varphi(v,x,\rho) \text{ is continuous.}
  \end{equation}
\end{lemma}
\begin{proof}
  We get the first bound by using~\eqref{w2'}, $|\varphi(x+Q(v,z,u,x,\rho ))-\varphi (x))|\le \|\nabla \varphi\|_\infty |Q(v,z,u,x,\rho )|$ and~\eqref{h6a}.
   From~\eqref{h6e} we have
  $\Lambda_{\varphi }(v',x,\rho')=\int_{E\times \R_{+}}(\varphi(x+Q_{v,x}(v',z,u,x,\rho' ))-\varphi (x))\mu (dz)du$ and thus
  \begin{align*}
    |\Lambda_{\varphi }(v,x,\rho)- \Lambda_{\varphi }(v',x,\rho')|&\le\|\nabla \varphi\|_\infty \int_{E\times \R_{+}}|Q(v,z,u,x,\rho )-Q_{v,x}(v',z,u,x,\rho' )| \mu(dz)du \\
    &\le L_\mu(c,\gamma)\|\nabla \varphi\|_\infty (|v-v'|+W_1(\rho,\rho')),
  \end{align*}
  by using~\eqref{h6d}.
  \\
  
  We now prove~\eqref{hyp_cont}.
  Let $(x_n)_{n\in \N}$ be a sequence in~$\R^d$ such that $x_n\to x$. We write:
  \begin{align*}
    \Lambda_\varphi(v,x_n,\rho)&= \int_{E\times \R_+}(\varphi(x_n+Q(v,z,u,x_n,\rho))-\varphi(x_n))\mu(dz)du\\
    &=\int_{E\times \R_+}(\varphi(x_n+Q(v,z,u,x_n,\rho))-\varphi(x_n+Q_{v,x_n}(v,z,u,x,\rho)))\mu(dz)du\\& +  \int_{E\times \R_+}(\varphi(x_n+Q_{v,x_n}(v,z,u,x,\rho))-\varphi(x_n))\mu(dz)du.
  \end{align*}
  By~\eqref{h6d},  the first integral is upper bounded by $\|\nabla \varphi\|_\infty L_\mu(c,\gamma)|x-x_n|\to 0.$ By~\eqref{h6e}, the second integral is equal to
  $$\int_{E\times \R_+}(\varphi(x_n+Q(v,z,u,x,\rho))-\varphi(x_n))\mu(dz)du. $$
We have $\varphi(x_n+Q(v,z,u,x,\rho))-\varphi(x_n) \to   \varphi(x+Q(v,z,u,x,\rho))-\varphi(x)$ and $|\varphi(x_n+Q(v,z,u,x,\rho))-\varphi(x_n)|\le \|\nabla\varphi\|_\infty |Q(v,z,u,x,\rho)|  $ that is $\mu(dz)du$-integrable. 
  The dominated convergence theorem gives then $ \Lambda_\varphi(v,x_n,\rho)\to  \Lambda_\varphi(v,x,\rho)$. 
\end{proof}

\begin{theorem}
  \label{Weq}Suppose that $\mathbf{(A)}$ holds. 
Then $\theta_{s,t}$, the flow given by Theorem~\ref{flow}, satisfy Equation (\ref{we2}).
\end{theorem}

\begin{proof}
  Let us consider $\rho \in \cP_1(\R^d)$ and the Euler scheme $\Theta _{s,s_{k}}^{\mathcal{P}}(\rho )$ associated to the partition $\mathcal{P}=\{s_{k}=s+\frac{k(t-s)}{n}:k=0,...,n\}.$ For $r\in \lbrack s_{k},s_{k+1})$ we denote $\tau(r)=s_{k}$ and associate the stochastic equation%
\begin{align}
X_{s,r}^{\mathcal{P}}=&X+\int_{s}^{r}b(X_{s,\tau (r^{\prime })}^{\mathcal{P}},\Theta _{s,\tau (r^{\prime })}^{\mathcal{P}}(\rho ))dr^{\prime} \label{we10}\\&+\int_{s}^{r}\int_{\R^{d}\times E\times \R_{+}}Q(v,z,u,X_{s,\tau (r^{\prime
})-}^{\mathcal{P}},\Theta _{s,\tau (r^{\prime })}^{\mathcal{P}}(\rho
))N_{\Theta ^{\mathcal{P}}}(dv,dz,du,dr^{\prime }), \notag
\end{align}
where $N_{\Theta ^{\mathcal{P}}}$ is a Poisson point measure of intensity $\Theta _{s,\tau (r^{\prime })}^{\mathcal{P}}(\rho )(dv)\mu (dz)dudr^{\prime}$ and $\mathcal{L}(X)=\rho .$ One has $\Theta _{s,s_{k}}^{\mathcal{P}}(\rho)(dx)=\mathcal{L}(X_{s,s_{k}}^{\mathcal{P}})$. Then, using It\^{o}'s formula in order to compute $\E(\varphi (X_{s,r}^{\mathcal{P}}))$ for $\varphi \in C_{b}^{1}({\mathbb{R}}^{d})$ we obtain %
\begin{align}
\int_{{\mathbb{R}}^{d}}\varphi (x)\Theta _{s,r}^{\mathcal{P}}(\rho
)(dx) &=\int_{{\mathbb{R}}^{d}}\varphi (x)\rho (dx)+\int_{s}^{
r}\int_{{\mathbb{R}}^{d}}\left\langle b(x,\Theta _{s,\tau (r^{\prime })}^{\mathcal{P}}(\rho )),\nabla \varphi (x)\right\rangle \Theta _{s,\tau (r')}^{\mathcal{P}}(\rho )(dx)dr'   \\
&+\int_{s}^{r}\int_{{\mathbb{R}}^{d}\times {\mathbb{R}}^{d}}\Lambda_{\varphi }(v,x,\Theta _{s,\tau (r)}^{\mathcal{P}}(\rho ))\Theta _{s,\tau
(r')}^{\mathcal{P}}(\rho )(dx)\Theta _{s,\tau (r')}^{\mathcal{P}}(\rho
)(dv)dr',  \notag
\end{align}
From~\eqref{h6b}, we get easily that $\E[|X_{s,r}^{\mathcal{P}}-X_{s,\tau(r)}^{\mathcal{P}}|]\le C/n$ for some constant $C\in \R_+^*$. Besides, we have $d_*(\theta_{s,r},\Theta _{s,r}^{\mathcal{P}})\le C(r-s)/n$ by Theorem~\ref{flow}. From~\eqref{lipb} and Lemma~\ref{lem_lambdaphi}, this leads to 
\begin{align*}
\int_{{\mathbb{R}}^{d}}\varphi (x)\Theta _{s,r}^{\mathcal{P}}(\rho
)(dx) &=\int_{{\mathbb{R}}^{d}}\varphi (x)\rho (dx)+\int_{s}^{
r}\int_{{\mathbb{R}}^{d}}\left\langle b(x,\theta_{s,r'}(\rho )),\nabla \varphi (x)\right\rangle \Theta _{s,\tau (r')}^{\mathcal{P}}(\rho )(dx)dr'  \\
&+\int_{s}^{r}\int_{{\mathbb{R}}^{d}\times {\mathbb{R}}^{d}}\Lambda_{\varphi }(v,x, \theta_{s,r'}(\rho ))\Theta _{s,\tau
(r')}^{\mathcal{P}}(\rho )(dx)\Theta _{s,\tau (r')}^{\mathcal{P}}(\rho
)(dv)dr' +R_n,  \notag
\end{align*}
with $|R_n|\le C/n$. Now, let us recall (see e.g.~\cite[Theorem 6.9]{Villani}) that $W_1(\rho_n,\rho_\infty)\to 0$ if, and only if, $\int_{\R^d} f(x) \rho_n(dx)\to\int_{\R^d} f(x) \rho_\infty(dx)$ for any continuous function $f:\R^d\to \R$ with sublinear growth (i.e. such that $\forall x, |f(x)|\le C(1+|x|)$ for some constant $C>0$). 
From~\eqref{sub_lin_A4} and~\eqref{hyp_cont} (resp.~\eqref{growthb} and~\eqref{lipb}), the map $(v,x)\mapsto \Lambda_\varphi(v,x,\theta_{s,r'}(\rho ))$ (resp. $x \mapsto \left\langle b(x,\theta_{s,r'}(\rho )),\nabla \varphi (x)\right\rangle$) is, for any $r'$, continuous with sublinear growth since $\varphi\in C^1_b(\R^d)$. Since $W_1(\Theta _{s,\tau (r')}^{\mathcal{P}}(\rho),\theta_{s,r'}(\rho))\to 0$ (and thus   $W_1(\Theta _{s,\tau (r')}^{\mathcal{P}}(\rho)\otimes\Theta _{s,\tau (r')}^{\mathcal{P}}(\rho),\theta_{s,r'}(\rho)\otimes\theta_{s,r'}(\rho))\to 0$), this gives the pointwise convergence for any~$r'$:
\begin{align*} &\int_{{\mathbb{R}}^{d}}\left\langle b(x,\theta_{s,r'}(\rho )),\nabla \varphi (x)\right\rangle \Theta _{s,\tau (r')}^{\mathcal{P}}(\rho )(dx) \to \int_{{\mathbb{R}}^{d}}\left\langle b(x,\theta_{s,r'}(\rho )),\nabla \varphi (x)\right\rangle \theta _{s,r'}(\rho )(dx),\\
&\int_{{\mathbb{R}}^{d}\times {\mathbb{R}}^{d}}\Lambda_{\varphi }(v,x, \theta_{s,r'}(\rho ))\Theta _{s,\tau
(r')}^{\mathcal{P}}(\rho )(dx)\Theta _{s,\tau (r')}^{\mathcal{P}}(\rho
)(dv) \to \int_{{\mathbb{R}}^{d}\times {\mathbb{R}}^{d}}\Lambda_{\varphi }(v,x, \theta_{s,r'}(\rho ))\theta _{s,r'}(\rho )(dx)\theta _{s,r'}(\rho
)(dv).\end{align*}
From the standard uniform bounds on the first moment, we can then get the claim by the dominated convergence theorem.
\end{proof}

\subsection{Probabilistic representation} In this subsection, we are interested in the existence and uniqueness of a process $(X_r,r\ge 0)$ on a probability space such that
    \begin{equation}\label{Prob_rep}
      X_t=X+\int_0^tb(X_r,\cL(X_r))dr+\int_0^t Q(v,z,u,X_{r-},\cL(X_{r-})) N(dv,dz,du,dr),    \end{equation}
    where $N$ is a Poisson point measure with intensity  $\cL(X_{r-})(dv)\mu(dz)dudr$. If such a process exists, we call it ``Boltzmann process''. We notice that the sublinear growth properties~\eqref{growthb} and~\eqref{h6a} gives $\sup_{t\in[0,T]}\E[|X_t|]<\infty$ for any $T>0$ and then
    \begin{equation}\label{lipsch_conti}\forall T>0,\exists C_T>0,\forall t\in[0,T], \E[|X_{t+h}-X_t|]\le C_T h,
    \end{equation}
   which yields to  $\cL(X_{t-})=\cL(X_t)$ for any $t\ge 0$. The existence of a solution to~\eqref{Prob_rep} is classically related to Martingale problems, see Horowitz and Karandikar~\cite{HoKa} in the context of the Boltzmann equation.  Our goal here is to underline some relations between the flow~$\theta$ introduced by Theorem~\ref{flow} and Equation~\eqref{Prob_rep}.

    From It\^o's Formula, it is clear that a Boltzmann process leads to a solution of the weak equation~\eqref{we2bis}. In Theorem~\ref{Weq}, we have proved under suitable conditions that the flow constructed by Theorem~\ref{flow} is a weak solution of~\eqref{we2bis}. Here, we prove that there exists a Boltzmann process that has the marginal laws given by the flow~$\theta_{0,t}$. 
\begin{theorem}
  \label{ExistenceRepr} Suppose that $\mathbf{(A)}$ holds. Then, there exists a process~$X$ that satisfies~\eqref{Prob_rep} such that $\cL(X_t)=\theta_{0,t}(\rho)$. 
\end{theorem}
 To prove this result, we consider the Euler scheme for which we know the convergence of the marginal laws by Theorem~\ref{flow}. We show by classical arguments that it gives a tight sequence in the Skorohod space and that any converging subsequence leads to a solution of the Martingale problem of~\eqref{Prob_rep}. 
 \begin{proof}
   Let $X\sim \rho $ with $\rho \in \mathcal{P}_{1}({\mathbb{R}}^{d})$.
We consider the time grid $t_{k}=\frac{k}{n}$, $k\in \N$ and we denote $\tau(t)=\frac{k}{n}$ for $\frac{k}{n}\leq t<\frac{k+1}{n}$. For $t\in \R$, we denote \begin{equation}\label{def_theta_n}\Theta^n_{0,t}=\Theta_{\frac{\lfloor nt \rfloor }{n},t} \Theta_{\frac{\lfloor nt \rfloor -1}{n},\frac{\lfloor nt \rfloor }{n}}\dots \Theta_{0,\frac 1n},
\end{equation}
with $\Theta_{s,t}$ defined by~\eqref{W3'}, so that $\Theta^n_{0,t}=\Theta^{\cP}_{0,t}$ with the partition $\cP=\{t_0<\dots<t_{\lfloor nt \rfloor}\le t\}$.
Then, we define the corresponding Euler scheme 
\begin{equation}
X_{t}^{n}=X+\int_{0}^{t}b(X_{\tau(r)}^{n},\Theta^n_{0,\tau(r)}(\rho) )dr+\int_{0}^{t}\int_{{\mathbb{R}}^{d}\times E\times {\mathbb{R}}
_{+}}Q(v,z,X_{\tau (r)},\Theta^n_{0,\tau(r)}(\rho))N_{\Theta^n}(dv,dz,du,dr),
\label{App2.4}
\end{equation}%
where $N_{\Theta^n}$ is a Poisson point measure with compensator $\Theta^n_{0,\tau (r)}(\rho)(dv)\mu (dz)dudr$ that is independent of~$X$. By construction, we have $\cL(X^n_t)=\Theta^n_{0,t}(\rho)$ for all $t\ge 0$. Theorem~\ref{flow} gives that there exists a flow $\theta_{s,t}$ corresponding to~$\Theta_{s,t}$, and we have
\begin{equation}\label{estim_unif_thetan}
  d_*(\theta_{0,t},\Theta^n_{0,t})\le \frac{C_2t}{n}.
\end{equation}

We first write the martingale problem associated with~$X^n$. 
 For $\varphi \in C_{b}^{1}(\R^{d})$, we define 
\begin{equation}
M^{n}_\varphi(t):=\varphi (X_{t}^{n})-\varphi
(X)-I_{t}^{n}-J_{t}^{n}  \label{App2.5}
\end{equation}%
with%
\begin{eqnarray*}
I_{t}^{n} &=&\int_{0}^{t} \int_{{\mathbb{R}}^{d}}\Lambda_{\varphi} (v,X_{r}^{n},\Theta^n_{0,\tau (r)}(\rho)))\Theta^n_{0,\tau (r)}(\rho)(dv)dr\\ J_{t}^{n} &=&\int_{0}^{t}\int_{{\mathbb{R}}^{d}}\left\langle
b(X_{\tau (r)}^{n},\Theta^n_{0,\tau (r)}(\rho))),\nabla \varphi (X_{r}^{n})\right\rangle dr .
\end{eqnarray*}%
This is a martingale, and we have for every $0\le s_{1}<...<s_{m}<t<t^{\prime }$ and every $\psi _{j}\in C_{b}^{0}(\R^{d})$
\begin{equation}
\E \left(\prod_{j=1}^{m}\psi _{j}(X_{\tau (s_{j})}^{n})M_{\varphi }^{n}(\tau(t^{\prime })) \right)=\E\left(\prod_{j=1}^{m}\psi _{j}(X_{\tau (t_{j})}^{n})M_{\varphi}^{n}(\tau (t)) \right).  \label{App2.6}
\end{equation}

We now analyse the convergence when $n\to \infty$ and  denote $P_{n}$ the probability measure on the
Skorohod space $\mathbb{D}(\R_{+},\R^{d})$ produced by the law of $X^{n}$. We easily check
Aldous' criterion : from~\eqref{growthb} and~\eqref{h6a}, we get a uniform bound on the first moment
\begin{equation}\label{uni_borne2}
  \forall T>0, \E\left[ \sup_{n\ge1} \sup_{t\in [0,T]} |X^n_r| \right]<\infty
\end{equation}
and then \begin{equation}\label{uni_borne3}
  \forall T>0, \exists C_T\in \R_+^*, \forall h\in [0,1], \E[\sup_{t\le T}|X^n_{t+h}-X^n_{t}|]\le C_Th.
\end{equation}
Thus, we obtain that the sequence $(P_{n})_{n\in \N}$ is
tight. Let $P$ be any limit point of this sequence. Up to consider a subsequence, we may assume that  $(P_{n})_{n\in \N}$ weakly converges to~$P$. We denote by $X_{t}$ the
canonical projections on $\mathbb{D}(\R_{+},\R^{d})$. We define 
\begin{equation}
M_{\varphi }(t):=\varphi (X_{t})-\varphi (X)-J_{t}-I_{t}
\label{App2.7}
\end{equation}%
with%
\begin{eqnarray*}
  I_{t} &=&\int_{0}^{t}\int_{{\mathbb{R}}^{d}}\Lambda _{\varphi}(v,X_{r},\theta_{0,r}(\rho))\theta_{0,r}(\rho)(dv)dr,\\
J_{t} &=&\int_{0}^{t}\int_{{\mathbb{R}}^{d}}\left\langle b(X_{r},\theta_{0,r}(\rho)),\nabla \varphi (X_{r})\right\rangle dr. 
\end{eqnarray*}%

\medskip
We now prove that $\E \left(\prod_{j=1}^{m}\psi _{j}(X_{\tau (s_{j})}^{n})M_{\varphi }^{n}(\tau(t^{\prime })) \right) \to \E_P \left(\prod_{j=1}^{m}\psi _{j}(X_{s_{j}})M_{\varphi }(t^{\prime }) \right)$, where $\E_P$ denotes the integration on $\mathbb{D}(\R_{+},\R^{d})$ with respect to~$P$. This then gives from~\eqref{App2.6}
\begin{equation}
\E\left(\prod_{j=1}^{m}\psi _{j}(X_{s_{j}})M_{\varphi }(t^{\prime})\right)=\E\left(\prod_{j=1}^{m}\psi _{j}(X_{s_{j}}) M_{\varphi }(t) \right).
\label{App2.8}
\end{equation}%
We define the intermediary terms
$$\hat{I}^n_t=\int_{0}^{t}\int_{{\mathbb{R}}^{d}}\Lambda _{\varphi}(v,X^n_{r},\theta_{0,r}(\rho))\theta_{0,r}(\rho)(dv)dr \text{ and } \hat{J}^n_t=\int_{0}^{t}\int_{{\mathbb{R}}^{d}}\left\langle b(X^n_{\tau(r)},\theta_{0,r}(\rho)),\nabla \varphi (X^n_{r})\right\rangle dr.$$
From the Lipschitz property of $b$ and $\Lambda_\varphi$ (Lemma~\ref{lem_lambdaphi}), we get
$$|I^n_{t}-\hat{I}^n_{t}|+|J^n_{t}-\hat{J}^n_{t}|\le C \int_0^{t}W_1(\theta_{0,r}(\rho),\Theta^n_{0,\tau(r)}(\rho))dr \underset{n\to \infty}\to 0,$$
by~\eqref{estim_unif_thetan} and~\eqref{uni_borne3}
Thus, it is sufficient to check the convergence of $$\E \left(\prod_{j=1}^{m}\psi _{j}(X_{\tau (s_{j})}^{n})[\varphi(X_{\tau (t')}^{n})-\varphi(X)-\hat{I}^n_{\tau (t')}-\hat{J}^n_{\tau (t')}]\right).$$
From~\eqref{sub_lin_A4}, $\Lambda_\varphi$ has a sublinear growth and is continuous with respect to~$x$. Therefore, $\mathbb{D}(\R_{+},\R^{d}) \ni \bar{x}\mapsto \int_0^{t'} \int_{\R^d} \Lambda_\varphi(v,\bar{x}(r),\theta_{0,r}(\rho))\theta_{0,r}(\rho)(dv)dr$ is continuous and bounded by $C(1+\sup_{r\in[0,t']} |\bar{x}(r)|+\sup_{r\in[0,t']}\int_{\R^d}|z|\theta_{0,r}(\rho)(dz))$ for some $C\in \R_+^*$. Since $ \E\left[ \sup_{n\ge1} \sup_{t\in [0,t']} |X^n_r| \right]<\infty$ by~\eqref{uni_borne2} and $P_n$ weakly converges to $P$, this gives the desired convergence for $s_1,\dots,s_m,t,t'\in[0,T] \setminus \mathfrak{D}$, where $\mathfrak{D}$ is an at most countable subset of $(0,T)$ (see Billingsley~\cite[p. 138]{Billingsley}). Last, from the right continuity under~$P$, we get that \eqref{App2.8} holds for any $0<s_1<\dots <s_m<t<t'$, which shows that $P$ is a solution of the Martingale Problem, i.e. is such that for any $\varphi \in C^1_b(\R^d)$, $M_\varphi(t)$ defined by~\eqref{App2.7} is a Martingale. Besides, let us notice that for all $t\ge 0$, $\cL(X_t)=\theta_{0,t}(\rho)$ by using~\eqref{estim_unif_thetan}.

\medskip

The classical theory of martingale problems allows to obtain
now the Equation (\ref{Prob_rep}). Let us be more explicit.
Let us denote by $\mu ^{X}$ the random point measure associated to the jumps of $%
X_{t}$.
Then, Theorem 2.42 in \cite{[JS]} guarantees that, as a solution of the
martingale problem, $X$ is a semimartingale with
characteristics $B_{r}=b(X_{r},\theta_{0,r}(\rho))$ and $\nu$ defined by $\nu ((0,t)\times A)=\int_{0}^{t}\int_{{\mathbb{R}}^{d}}\int_{E\times \R_{+}}1_{A}(Q(v,z,X_{r},\theta_{0,r}(\rho)))\mu (dz)du\theta_{0,r}(\rho)(dv)dr$. Then, by
Theorem 2.34 in \cite{[JS]}  one has 
\begin{equation*}
X_{t}=X+\int_{0}^{t}b(X_{r},\theta_{0,r}(\rho))dr+\int_{0}^{t}y\mu^{X}(dr,dy)
\end{equation*}%
and the compensator of $\mu^{X}$ is $\nu$ (in \cite{[JS]} a truncation function $h$ appears, here we take
the truncation function $h(x)=0$, which  is possible because we work in the
framework of finite variation $\int \left\vert x\right\vert \theta_{0,r}(\rho)(dx)<\infty $). Then, using the representation given in
\cite[Theorem 7.4, p. 93]{[IW]}, one may construct a probability space and a
Poisson point measure $N_{\theta}$ of compensator $\mu (dz)du \theta_{0,r}(\rho)(dv)dr$ such
that the process
\begin{equation*}
\overline{X}_{t}=X_{0}+\int_{0}^{t}b(\overline{X}_{r},\theta_{0,r}(\rho))dr+\int_{0}^{t}\int_{{\mathbb{R}}^{d}}\int_{E\times \R_{+}}Q(v,z,\overline{X}_{r-},\theta_{0,r}(\rho)))N_{\theta}(dr,dz,du,dv)
\end{equation*}%
has the same law as $X$. Since $\cL(\overline{X}_{t})=\cL(X_t)=\theta_{0,t}(\rho)$ is continuous with respect to~$t$, this produces a solution of~\eqref{Prob_rep}.
\end{proof}

 Theorem~\ref{ExistenceRepr} gives the existence of a Boltzmann process~\eqref{Prob_rep} such that $\cL(X_t)=\theta_{0,t}$. It would be interesting to have a uniqueness result and get for example that the marginal laws of any process satisfying~\eqref{Prob_rep} are given by $\theta_{0,t}$, $t\ge 0$.   Unfortunately, this does not seem possible to prove such a result in our general framework, and we thus state it with a standard Lipschitz assumptions for~$Q$.

 \begin{proposition}
   Let us assume that $\mathbf{(A)}$ holds and that $Q$ satisfies the following Lipschitz assumption:
  \begin{align}\label{std_lip}
&\int_{E\times {\mathbb{R}}_{+}}\left\vert Q(v_{1},z,u,x_{1},\rho
_{1})-Q(v_{2},z,u,x_{2},\rho _{2})\right\vert \mu (dz)du  \leq L_\mu(c,\gamma) (\left\vert x_{1}-x_{2}\right\vert +\left\vert
v_{1}-v_{2}\right\vert +W_{1}(\rho _{1},\rho _{2})),
  \end{align}
i.e.~\eqref{h6d} is true with $Q_{v,x}=Q$.  Then, any process $(X_t, t\ge 0)$ that satisfies~\eqref{Prob_rep} is such that $\cL(X_t)=\theta_{0,t}(\rho)$ for all $t\ge 0$. 
\end{proposition}
 \begin{proof}
   Let $X$ be a solution of~\eqref{Prob_rep}. We denote $f_t=\cL(X_t)$. There exists a Poisson point measure $N$ with intensity $f_r(dv)\mu(dz)dvdr$ such that
   $$X_t=X+\int_0^t b(X_r,f_r)dr +\int_0^tQ(v,z,u,X_{r-},f_r)N(dv,dz,du,dr).$$
   As for the preceding proof, we consider the time grid $t_{k}=\frac{k}{n}$, $k\in \N$ and  denote $\tau(t)=\frac{k}{n}$ for $\frac{k}{n}\leq t<\frac{k+1}{n}$. We define the process $X^n$ by:
   $$X^n_t=X+\int_0^t b(X^n_{\tau(r)},f_{\tau(r)})dr +\int_0^tQ(v,z,u,X_{\tau(r)},f_{\tau(r)})N(dv,dz,du,dr).$$
   We have 
   \begin{align*}
     |X_t-X^n_t|\le &  \int_0^t |b(X_r,f_r)-b(X^n_{\tau(r)},f_{\tau(r)})|dr \\&+\int_0^t |Q(v,z,u,X_{r-},f_r)-Q(v,z,u,X_{\tau(r)},f_{\tau(r)})|N(dv,dz,du,dr).
   \end{align*}
   By using~\eqref{lipb} and~\eqref{std_lip}, we get
   $$\E[|X_t-X^n_t|]\le \int_0^t (L_b+L_\mu(c,\gamma)) \left(\E[|X_r-X^n_{\tau(r)}|]+W_1(f_r,f_{\tau(r)})\right)dr.$$
   Now, we observe that $W_1(f_r,f_{\tau(r)})\le \frac{C_T}{n}$ for $r\in[0,T]$ by using~\eqref{lipsch_conti}. Similarly, we observe from the sublinear growth properties~\eqref{growthb} that for any $T>0$,
   $\E\left[ \sup_{n\ge1} \sup_{t\in [0,T]} |X^n_r| \right]<\infty$ and then
   \begin{equation}\label{lipW1_eul}  \exists C_T\in \R_+^*, \forall h\in [0,1], \E[\sup_{t\le T}|X^n_{t+h}-X^n_{t}|]\le C_Th.
   \end{equation}
   We therefore get for any $T>0$ the existence of a constant $C_T$ such that
   $\E[|X_t-X^n_t|]\le \int_0^t (L_b+L_\mu(c,\gamma)) \E[|X_r-X^n_{r}|]dr+ \frac{C_T}n$, and then
   \begin{equation}\label{cv_Euler} \E[|X_t-X^n_t|]\le  \frac{C_T}n \exp((L_b+L_\mu(c,\gamma))t), \ t\in [0,T],
   \end{equation}
   by Gronwall lemma. This gives $W_1(\cL(X^n_t),f_t)\underset{n\to \infty}\to 0$.

   On the other hand, we get by using Lemma~\ref{STABILITY} that
 \begin{align*}W_1(\cL(X^n_{t_{k+1}}),\Theta^n_{0,t_{k+1}}(\rho))\le  W_1(\cL(X^n_{t_{k}}),\Theta^n_{0,t_{k}}(\rho))&\left(1+2\frac{L_\mu(c,\gamma)+L_b}{n} \right)\\&+L_\mu(c,\gamma)\int_{t_k}^{t_{k+1}} W_1( f_t,\Theta^n_{0,t_{k}}(\rho))dt ,
 \end{align*}
   where $\Theta^n_{0,t}$ is defined by~\eqref{def_theta_n}. For $t_{k+1}\le T$ and $t\in[t_k,t_{k+1}]$, we have $W_1( f_t,\Theta^n_{0,t_{k}}(\rho))\le W_1( f_t,f_{t_k})+W_1( f_{t_k},\Theta^n_{0,t_{k}}(\rho))\le \frac{C_T}{n}$ for some constant $C_T$ by using~\eqref{lipsch_conti} and~\eqref{cv_Euler}. Therefore, we get for $t_{k+1}\le T$ that
   $$W_1(\cL(X^n_{t_{k+1}}),\Theta^n_{0,t_{k+1}}(\rho))\le  W_1(\cL(X^n_{t_{k}}),\Theta^n_{0,t_{k}}(\rho))\left(1+2\frac{L_\mu(c,\gamma)+L_b}{n} \right)+\frac{C_T}{n^2},$$
   for some constant $C_T>0$. Since $\cL(X^n_0)=\Theta^n_{0,0}(\rho)=\rho$, we get for  $t_k \in [0,T]$:
   $$  W_1(\cL(X^n_{t_{k}}),\Theta^n_{0,t_{k}}(\rho)) \le \frac{C_T}{n^2} \frac{\left(1+2\frac{L_\mu(c,\gamma)+L_b}{n} \right)^k -1}{2\frac{L_\mu(c,\gamma)+L_b}{n}} \le  \frac{C_T}{2(L_\mu(c,\gamma)+L_b)n} \exp( 2(L_\mu(c,\gamma)+L_b)T).$$
   Since $T>0$ is arbitrary, we obtain by using this bound together with~\eqref{lipW1_eul} and Theorem~\ref{flow} that $W_1(\cL(X^n_{t}),\theta_{0,t}(\rho))\underset{n\to \infty} \to  0$ for any~$t\ge 0$. This shows that $f_t=\theta_{0,t}(\rho)$ since we already have proven that $W_1(\cL(X^n_{t}),f_t)\underset{n\to \infty} \to  0$.
 \end{proof}

\section{Particle system approximation} \label{particles}

Particle systems have been used for a long time to show existence results on nonlinear SDE of McKean-Vlasov and Boltzmann type, see e.g. Sznitman~\cite{Sznitman} or M\'el\'eard~\cite{Meleard}. Formally, the interacting particle system associated to the equation~\eqref{Prob_rep} can be written as follows:
$$X^i_t=X^i_0+\int_0^tb\left(X^i_r,\frac 1N \sum_{j=1}^N\delta_{X^j_{r-}}\right)dr+\int_0^t\int_{\R^d\times E \times \R_+}Q\left(v,z,u,X^i_{r-},\frac 1N \sum_{j=1}^N\delta_{X^j_{r-}}\right)N^i(dv,dz,du,dr),$$
where $N^i(dv,dz,du,dr)$, $i=1,\dots,N$, are independent Poisson point measure with intensity $\left(\frac 1N \sum_{j=1}^N\delta_{X^j_{r-}}(dv)\right)\mu(dz)dudr$. In this section, we do not discuss this interacting particle system itself, but we focus on its discretization.

More precisely, we are interested in the approximation of operator $\Theta_{s,t}$ defined by~\eqref{W3'} and of the corresponding Euler scheme. Particle systems gives then a tool to approximate~$\Theta_{s,t}$ and thus the flow $\theta_{s,t}$ defined by Theorem~\ref{flow}.  Through this section,  
we will work with the space $\mathcal{P}(\mathcal{P}_{1}({\mathbb{R}}^{d}))$ of the probability measures on $\mathcal{P}_{1}({%
\mathbb{R}}^{d})$\ (with the Borel $\sigma $ field associated to the
distance $W_{1}$). We denote by $\mathcal{P}_{1}(\mathcal{P}_{1}({\mathbb{R}}%
^{d}))$\ the space of probability measures $\eta \in \mathcal{P}(\mathcal{P}%
_{1}({\mathbb{R}}^{d}))$\ such that 
\begin{equation*}
\int_{\mathcal{P}_{1}({\mathbb{R}}^{d})}W_{1}(\mu ,\delta _{0})\eta (d\mu
)<\infty .
\end{equation*}%
On $\mathcal{P}_{1}(\mathcal{P}_{1}({\mathbb{R}}^{d}))$, we take the
Wasserstein distance 
\begin{eqnarray*}
\mathcal{W}_{1}(\eta _{1},\eta _{2}) &=&\inf_{\pi \in \Pi (\eta _{1},\eta
_{2})}\int_{\mathcal{P}_{1}({\mathbb{R}}^{d})\times \mathcal{P}_{1}({\mathbb{%
R}}^{d})}W_{1}(\mu ,\nu )\pi (d\mu ,d\nu ) \\
&=&\sup_{L(\Phi )\leq 1}\left\vert \int_{\mathcal{P}_{1}({\mathbb{R}}%
^{d})}\Phi (\mu )\eta _{1}(d\mu )-\int_{\mathcal{P}_{1}({\mathbb{R}}%
^{d})}\Phi (\mu )\eta _{2}(d\mu )\right\vert
\end{eqnarray*}%
where $\Pi (\eta _{1},\eta _{2})$ is the set of probability measures on $%
\mathcal{P}_{1}({\mathbb{R}}^{d})\times \mathcal{P}_{1}({\mathbb{R}}^{d})$
with marginals $\eta _{1}$ and $\eta _{2}$ and $L(\Phi )$ is the Lipschitz
constant of $\Phi ,$ so that $\left\vert \Phi (\mu )-\Phi (\nu )\right\vert
\leq L(\Phi )W_{1}(\mu ,\nu ).$ Before going on, we list some basic
properties of $\mathcal{W}_{1}$ which will be used in the following. First
we notice that $\Pi (\eta ,\delta _{\mu })=\{\eta \otimes \delta _{\mu }\}$
(the product probability of $\eta $ and $\delta _{\mu }$\ is the only
probability measure on the product space which has the marginals $\eta $ and 
$\delta _{\mu }).$ As an immediate consequence, we have
$$W_1(\eta,\delta_\mu)=\int_{\mathcal{P}_{1}({\mathbb{R}}^{d})}W_{1}(\nu ,\mu)\eta (d\nu )$$
and $\mathcal{W}_{1}(\delta _{\mu
},\delta _{\nu })=W_{1}(\mu ,\nu ).$ Another fact, used in the following, is
that for every $\eta \in \mathcal{P}_{1}(\mathcal{P}_{1}({\mathbb{R}}^{d}))$%
\ and $\mu _{0}\in \mathcal{P}_{1}({\mathbb{R}}^{d})$\ 
\begin{eqnarray}
\int_{\mathcal{P}_{1}({\mathbb{R}}^{d})}\left\vert \int_{{\mathbb{R}}%
^{d}}f(x)\nu (dx)-\int_{{\mathbb{R}}^{d}}f(x)\mu _{0}(dx)\right\vert \eta
(d\nu ) &\leq &L(f)\int_{\mathcal{P}_{1}({\mathbb{R}}^{d})}W_{1}(\nu ,\mu
_{0})\eta (d\nu )  \label{chain1b} \\
&=&L(f)\mathcal{W}_{1}(\eta ,\delta _{\mu _{0}}).  \notag
\end{eqnarray}

The main object in this section is a random vector $%
X=(X^{1},....,X^{N}),X^{i}\in {\mathbb{R}}^{d}$, $i=1,...,N,$ where the
dimension $N$ is given (fixed). We assume that $\E(\left\vert
X^{i}\right\vert )<\infty $ and we associate the (random) empirical measure
on ${\mathbb{R}}^{d}$ 
\begin{equation*}
\widehat{\rho }(X)(dv)=\frac{1}{N}\sum_{i=1}^{N}\delta _{X^{i}}(dv).
\end{equation*}%
Notice that $\widehat{\rho }(X)$ is a random variable with values in $%
\mathcal{P}_{1}({\mathbb{R}}^{d})$ so the law $\mathcal{L}(\widehat{\rho }%
(X))\in \mathcal{P}(\mathcal{P}_{1}({\mathbb{R}}^{d}))$ and, for every $\Phi
:\mathcal{P}_{1}({\mathbb{R}}^{d})\rightarrow {\mathbb{R}}_{+}$%
\begin{equation*}
\int_{\mathcal{P}_{1}({\mathbb{R}}^{d})}\Phi (\mu ) \mathcal{L}(\widehat{%
\rho }(X))(d\mu )={\mathbb{E}}(\Phi (\widehat{\rho }(X))).
\end{equation*}%
In particular, we have%
\begin{eqnarray}
\mathcal{W}_{1}(\mathcal{L}(\widehat{\rho }(X)),\mathcal{L}(\widehat{\rho }%
(Y))) &=&\sup_{L(\Phi )\leq 1}\left\vert {\mathbb{E}}(\Phi (\widehat{\rho }%
(X)))-{\mathbb{E}}(\Phi (\widehat{\rho }(Y)))\right\vert  \label{chain1a} \\
&\leq &{\mathbb{E}}(W_{1}(\widehat{\rho }(X),\widehat{\rho }(Y)))  \notag \\
&\leq &\frac{1}{N}\sum_{i=1}^{N}{\mathbb{E}}(\left\vert
X^{i}-Y^{i}\right\vert ).  \notag
\end{eqnarray}%
This also proves by taking $Y^{i}=0$ that $\mathcal{L}(\widehat{\rho }%
(X))\in \mathcal{P}_{1}(\mathcal{P}_{1}({\mathbb{R}}^{d})).$

In the following we consider an initial vector $X_{0}$ and will assume that
the components $X_{0}^{1},...,X_{0}^{N}$ are identically distributed and we
denote $\rho \in \mathcal{P}_{1}({\mathbb{R}}^{d})$ the common law: $%
X_{0}^{i}\sim \rho ,i=1,...,N$. We consider the uniform grid $%
s=s_{0}<....<s_{n}=t,s_{i}=s+\frac{i}{n}(t-s)$ and we construct two
sequences $X_{k}$ and $\overline{X}_{k},k=0,1,...,n$ in the following way.
We start with $\overline{X}_{0}=X_0$. The sequence $X_{k}$ (respectively $\overline{X}_{k}$) is
constructed by using the empirical measures $\widehat{\rho }(X_{k})$
(respectively the measure $\Theta_{s,s_k}^n(\rho):=\Theta_{s_{k-1},s_k}\dots \Theta_{s,s_1}(\rho)$ with $\Theta_{s,t}(\rho)$ defined by~\eqref{W3'}), and we define by recurrence:
\begin{align}
  X_{k+1}^{i} =&X_{k}^{i}+b(X_{k}^{i},\widehat{\rho }(X_{k}))(s_{k+1}-s_{k})  \label{chain5} \\
  &+%
\int_{s_{k}}^{s_{k+1}}\int_{{\mathbb{R}}^{d}\times E\times {\mathbb{R}}%
_{+}}Q(v,z,u,X_{k}^{i},\widehat{\rho }(X_{k}))N_{\widehat{\rho }%
(X_{k})}^{i}(dv,dz,du,dr), \notag\\
\overline{X}_{k+1}^{i} =&\overline{X}_{k}^{i}+b(\overline{X}_{k}^{i},\Theta_{s,s_k}^n(\rho))(s_{k+1}-s_{k})\label{chain6}\\&+\int_{s_{k}}^{s_{k+1}}\int_{{\mathbb{R}}%
^{d}\times E\times {\mathbb{R}}_{+}}Q(v,z,u,\overline{X}_{k}^{i},\Theta_{s,s_k}^n(\rho))N_{\Theta_{s,s_k}^n(\rho)}^{i}(dv,dz,du,dr),\notag
\end{align}%
where $N_{\widehat{\rho }(X_{k})}^{i}(dv,dz,du,dr)$, $i=1,\dots ,N$ (resp. $%
N_{\Theta_{s,s_k}^n(\rho)}^{i}(dv,dz,du,dr)$) are Poisson point measures
that are independent each other conditionally to $X_{k}$ (resp. $\overline{X}_{k}$) with intensity $\widehat{\rho }(X_{k})(dv)\mu (dz)dudr$ (resp. $\Theta_{s,s_k}^n(\rho)(dv)\mu (dz)dudr$). Let us observe that the common law of $\overline{X}_{k}^{i}$, $i=1,...,N$, is 
\begin{equation}
\cL (\overline{X}^i_{k})=\Theta _{s,s_{k}}^{n}(\rho ).  \label{chain7}
\end{equation}%
Note that $\overline{X}^1_{k},\dots,\overline{X}^N_{k}$ are independent. 

\begin{theorem}
\label{theorem_approx} Assume that $(\mathbf{A})$ holds true and $%
X_{0}^{i},i=1,...,N$ are independent and of law $\rho\in \mathcal{P}_1({%
\mathbb{R}}^d)$. We assume that $M_q=\left(\int_{{\mathbb{R}}^d} |x|^q\rho(dx)\right)^{1/q}<\infty$ with $q>\frac{d}{d-1}\wedge 2$. We define 
\begin{equation*}
V_N=1_{d=1} N^{-1/2}+1_{d=2} N^{-1/2}\log(1+N)+1_{d\ge 3}N^{-1/d}.
\end{equation*}
Then there exists a constant~$C$ depending on $d$, $q$, $(b,c,\gamma)$ and $%
(t-s)$ such that for every Lipschitz function $f:{\mathbb{R}}^d\to {\mathbb{R%
}}$ with $L(f)\leq 1$%
\begin{equation}
{\mathbb{E}}\left(\left\vert \frac{1}{N}\sum_{i=1}^{N}f(X_{n}^{i})-\int_{{%
\mathbb{R}}^{d}}f(x)\Theta _{s,s_{n}}^{n}(\rho )(dx)\right\vert \right)\leq
C M V_N.  \label{chain8}
\end{equation}%
Besides, we have the propagation of chaos in Wasserstein distance:
$$W_1(\cL(X_n^1,\dots,X_n^m),\Theta _{s,s_{n}}^{n}(\rho )(dx)\otimes\dots\otimes \Theta _{s,s_{n}}^{n}(\rho )(dx))\le m C M V_N \underset{N\to \infty}{\to}0.$$
Furthermore, if $\theta_{s,t}$ denotes the flow given by Theorem~\ref{flow}, we
have 
\begin{equation}
{\mathbb{E}}\left(\left\vert \frac{1}{N}\sum_{i=1}^{N}f(X_{n}^{i})-\int_{{%
\mathbb{R}}^{d}}f(x)\theta _{s,t}(\rho )(dx)\right\vert \right)\leq CMV_N +%
\frac{C}{n}.  \label{chain8'}
\end{equation}
\end{theorem}

To prove Theorem~\ref{theorem_approx}, we introduce an intermediary sequence 
$\widetilde{X}_{k}^{i}$ defined as follows. On the first time step, we
define 
\begin{equation}
\widetilde{X}_{1}^{i}=X_{0}^{i}+b(X_{0}^{i},\rho
)(s_{1}-s_{0})+\int_{s_{0}}^{s_{1}}\int_{{\mathbb{R}}^{d}\times E\times {%
\mathbb{R}}_{+}}Q(v,z,u,X_{0}^{i},\rho )N_{\rho }^{i}(dv,dz,du,dr),
\label{chain2'}
\end{equation}%
so that $\widetilde{X}_{1}^{i}=\overline{X}_{1}^{i}$. Then, for the next
time steps, we define for $k\geq 1$, 
\begin{equation*}
\widetilde{X}_{k+1}^{i}=\widetilde{X}_{k}^{i}+b(\widetilde{X}_{k}^{i},%
\widehat{\rho }(\widetilde{X}_{k}))(s_{k+1}-s_{k})+\int_{s_{k}}^{s_{k+1}}%
\int_{{\mathbb{R}}^{d}\times E\times {\mathbb{R}}_{+}}Q(v,z,u,\widetilde{X}%
_{k}^{i},\widehat{\rho }(\widetilde{X}_{k}))N_{\widehat{\rho }(\widetilde{X}%
_{k})}^{i}(dv,dz,du,dr),\quad \quad
\end{equation*}%
where $N_{\widehat{\rho }(\widetilde{X}_{k})}^{i}(dv,dz,du,dr)$ is a Poisson
process with intensity $\widehat{\rho }(\widetilde{X}_{k})(dv)\mu (dz)dudr$.
We stress that for $k\geq 1,$ the intensity of the Poisson point measures
is, as for $X_{k+1}^{i}$, the empirical measure of the vector constructed in
the previous step. However, since $X_{1}\neq \widetilde{X}_{1},$ the two
chains are different.

\begin{lemma}
\label{lem_Xtilde} Let assumption ($\mathbf{A}$) hold. We assume that the
components of $X_{0}=(X_{0}^{1},...,X_{0}^{N})$ have the common distribution 
$\rho \in \mathcal{P}_{1}({\mathbb{R}}^{d})$. Then, we have 
\begin{align*}
&\mathcal{W}_{1}(\mathcal{L}(\widehat{\rho }(X_{n})),\mathcal{L}(\widehat{%
\rho }(\widetilde{X}_{n})))\leq \frac{e^{2L(t-s)}L(t-s)}{n}\mathcal{W}_{1}(\mathcal{L}(\widehat{  \rho }(X_{0})),\delta _{\rho }),\\
&W_1\left(\cL(X^1_n,\dots,X^m_n), \cL(\widetilde{X}^1_n,\dots,\widetilde{X}^m_n)\right)\leq m\frac{e^{2L(t-s)}L(t-s)}{n}
  \mathcal{W}_{1}(\mathcal{L}(\widehat{\rho }(X_{0})),\delta _{\rho }),
\end{align*}
with $L=L_{b}+2L(c,\gamma )$.
\end{lemma}
\begin{proof}
\textbf{Step 1} We first construct by recurrence the sequences $x_{k},%
\widetilde{x}_{k},k=1,...,n$ in the following way. We take $\pi _{0}(dv,d%
\overline{v})$ to be the optimal coupling of $\rho $ and of $\widehat{\rho }%
(X_{0})$ and we take $\tau _{0}:[0,1]\rightarrow {\mathbb{R}}^{d}\times {%
\mathbb{R}}^{d}$ that represents $\pi _{0}$. In particular, if $\tau
_{0}=(\tau _{0}^{1},\tau _{0}^{2})$ then $\tau _{0}^{1}$ represents $\widehat{\rho }(X_{0})$ and $\tau _{0}^{2}$ represents $\rho$. We
note that the optimality of $\pi _{0}$ gives $W_{1}(\rho ,\widehat{\rho }(X_{0}))=\int_{0}^{1}|\tau _{0}^{2}(u)-\tau _{0}^{1}(u)|du$ and thus 
\begin{equation}
\mathcal{W}_{1}(\delta _{\rho},\mathcal{L}(\widehat{\rho }(X_{0})))={%
\mathbb{E}}[W_{1}(\rho ,\widehat{\rho }(X_{0}))]={\mathbb{E}}\left[
\int_{0}^{1}|\tau _{0}^{2}(u)-\tau _{0}^{1}(u)|du\right] .  \label{calW1}
\end{equation}%
Then we define 
\begin{eqnarray*}
x_{1}^{i} &=&X_{0}^{i}+b(X_{0}^{i},\widehat{\rho }(X_{0}))(s_{1}-s_{0})+%
\int_{s_{0}}^{s_{1}}\int_{[0,1]\times E\times {\mathbb{R}}_{+}}Q(\tau
_{0}^{1}(w),z,u,X_{0}^{i},\widehat{\rho }(X_{0}))N^{i}(dw,dz,du,dr), \\
\widetilde{x}_{1}^{i} &=&X_{0}^{i}+b(X_{0}^{i},\rho
)(s_{1}-s_{0})+\int_{s_{0}}^{s_{1}}\int_{[0,1]\times E\times {\mathbb{R}}%
_{+}}Q_{\tau _{0}^{1}(w),X_{0}^{i}}(\tau _{0}^{2}(w),z,u,X_{0}^{i},\rho
)N^{i}(dw,dz,du,dr).
\end{eqnarray*}%
where $N^{i}$ is a Poisson process with intensity $1_{[0,1]}(w)dw\mu
(dz)dudr.$ We also assume that the Poisson point measures $N^{i},i=1,...,N$
are independent. Notice that $x_{1}$ has the same law as $X_{1}$ and $%
\widetilde{x}_{1}$ has the same law as $\widetilde{X}_{1}$ by (\ref{h6e}).

Then, for $k\geq 1,$ if $x_{k},\widetilde{x}_{k}$ are given, we construct $%
x_{k+1},\widetilde{x}_{k+1}$ as follows. We consider $\pi _{k}(dv,d\overline{%
v})$ an optimal coupling of $\widehat{\rho }(x_{k})$ and $\widehat{\rho }(%
\widetilde{x}_{k})$ and we take $\tau _{k}:[0,1]\rightarrow {\mathbb{R}}%
^{d}\times {\mathbb{R}}^{d}$ that represents $\pi _{k}$. 
Then we define 
\begin{eqnarray*}
x_{k+1}^{i} &=&x_{k}^{i}+b(x_{k}^{i},\widehat{\rho }(x_{k}))(s_{k+1}-s_{k})+%
\int_{s_{k}}^{s_{k+1}}\int_{[0,1]\times E\times {\mathbb{R}}_{+}}Q(\tau
_{k}^{1}(w),z,u,x_{k}^{i},\widehat{\rho }(x_{k}))N^{i}(dw,dz,du,dr), \\
\widetilde{x}_{k+1}^{i} &=&\widetilde{x}_{k}^{i}+b(\widetilde{x}_{k}^{i},%
\widehat{\rho }(\widetilde{x}_{k}))(s_{k+1}-s_{k})+\int_{s_{k}}^{s_{k+1}}%
\int_{[0,1]\times E\times {\mathbb{R}}_{+}}Q_{\tau
_{k}^{1}(w),x_{k}^{i}}(\tau _{k}^{2}(w),z,u,\widetilde{x}_{k}^{i},\widehat{%
\rho }(\widetilde{x}_{k}))N^{i}(dw,dz,du,dr)
\end{eqnarray*}%
Using again (\ref{h6e}), we get by induction on~$k$ that $x_{k}$ has the
same law as $X_{k}$ and $\widetilde{x}_{k}$ has the same law as $\widetilde{X%
}_{k}$. 

\textbf{Step 2.} Suppose that $k\geq 1.$ We use now Assumptions~\eqref{lipb}
and~\eqref{h6d} to get%
\begin{eqnarray*}
{\mathbb{E}}\left\vert x_{k+1}^{i}-\widetilde{x}_{k+1}^{i}\right\vert &\leq &%
{\mathbb{E}}\left\vert x_{k}^{i}-\widetilde{x}_{k}^{i}\right\vert
+(L_{b}+L_{\mu }(c,\gamma ))({\mathbb{E}}\left\vert x_{k}^{i}-\widetilde{x}%
_{k}^{i}\right\vert +W_{1}(\widehat{\rho }(x_{k}),\widehat{\rho }(\widetilde{%
x}_{k})))(s_{k+1}-s_{k}) \\
&&+L_{\mu }(c,\gamma ){\mathbb{E}}\int_{s_{k}}^{s_{k+1}}\int_{0}^{1}\left%
\vert \tau _{k}^{1}(\omega )-\tau _{k}^{2}(w)\right\vert dwds.
\end{eqnarray*}%
Since%
\begin{equation*}
\int_{0}^{1}\left\vert \tau _{k}^{1}(\omega )-\tau _{k}^{2}(w)\right\vert
dw=W_{1}(\widehat{\rho }(x_{k}),\widehat{\rho }(\widetilde{x}_{k}))\leq 
\frac{1}{N}\sum_{j=1}^{N}\left\vert x_{k}^{j}-\widetilde{x}%
_{k}^{j}\right\vert ,
\end{equation*}%
we obtain%
\begin{eqnarray}\label{ineq_rec}
{\mathbb{E}}\left\vert x_{k+1}^{i}-\widetilde{x}_{k+1}^{i}\right\vert &\leq &%
{\mathbb{E}}\left\vert x_{k}^{i}-\widetilde{x}_{k}^{i}\right\vert
[1+(L_{b}+L_{\mu }(c,\gamma ))(s_{k+1}-s_{k})] \\
&&+(L_{b}+2L_{\mu }(c,\gamma ))(s_{k+1}-s_{k})\frac{1}{N}\sum_{j=1}^{N}{%
\mathbb{E}}|x_{k}^{j}-\widetilde{x}_{k}^{j}|. \notag
\end{eqnarray}%
Summing over $i=1,...,N$, we get 
\begin{equation*}
\frac{1}{N}\sum_{i=1}^{N}{\mathbb{E}}\left\vert x_{k+1}^{i}-\widetilde{x}%
_{k+1}^{i}\right\vert \leq \frac{1}{N}\sum_{i=1}^{N}{\mathbb{E}}\left\vert
x_{k}^{i}-\widetilde{x}_{k}^{i}\right\vert [1+2(L_{b}+2L_{\mu }(c,\gamma
))(s_{k+1}-s_{k})].
\end{equation*}%
Using this inequality, we get by recurrence 
\begin{equation*}
\frac{1}{N}\sum_{i=1}^{N}{\mathbb{E}}\left\vert x_{n}^{i}-\widetilde{x}%
_{n}^{i}\right\vert \leq \frac{1}{N}\sum_{i=1}^{N}{\mathbb{E}}\left\vert
x_{1}^{i}-\widetilde{x}_{1}^{i}\right\vert \left( 1+\frac{2(L_{b}+2L_{\mu
}(c,\gamma ))}{n}(s-t)\right) ^{n-1}.
\end{equation*}

\textbf{Step 3} We go now from $x_{1},\widetilde{x}_{1}$ to $X_{0}$. We have 
\begin{eqnarray*}
{\mathbb{E}}\left\vert x_{1}^{i}-\widetilde{x}_{1}^{i}\right\vert &\leq
&L_{\mu }(c,\gamma )\int_{s}^{s_{1}}{\mathbb{E}}\int_{0}^{1}\left\vert \tau
_{0}^{1}(\omega )-\tau _{0}^{2}(w)\right\vert dwds+(L_{b}+L_{\mu }(c,\gamma
))W_{1}(\rho ,\widehat{X}_{0})(s_{1}-s) \\
&=&(L_{b}+2L_{\mu }(c,\gamma ))(s_{1}-s)\mathcal{W}_{1}(\delta _{\rho },%
\mathcal{L}(\widehat{\rho }(X_{0})))
\end{eqnarray*}%
by using~\eqref{calW1} for the last equality, so that we get by summing over 
$i=1,...,N$, 
\begin{equation*}
\frac{1}{N}\sum_{i=1}^{N}{\mathbb{E}}\left\vert x_{1}^{i}-\widetilde{x}%
_{1}^{i}\right\vert \leq (L_{b}+2L_{\mu }(c,\gamma ))\frac{t-s}{n}\mathcal{W}%
_{1}(\delta _{\rho },\mathcal{L}(\widehat{\rho }(X_{0}))).
\end{equation*}%
We combine with the previous inequality and we obtain%
\begin{eqnarray*}
\frac{1}{N}\sum_{i=1}^{N}\E\left\vert x_{n}^{i}-\widetilde{x}%
_{n}^{i}\right\vert &\leq &\mathcal{W}_{1}(\delta _{\rho },\mathcal{L}(%
\widehat{\rho }(X_{0})))\frac{(L_{b}+2L_{\mu }(c,\gamma ))(t-s)}{n}\left( 1+%
\frac{L_{b}+2L_{\mu }(c,\gamma )}{n}(t-s)\right) ^{n-1} \\
&\leq &\frac{e^{(L_{b}+2L(c,\gamma ))(t-s)}(L_{b}+2L_{\mu }(c,\gamma ))(t-s)%
}{n}\mathcal{W}_{1}(\delta _{\rho },\mathcal{L}(\widehat{\rho }(X_{0}))).
\end{eqnarray*}
We notice that the law of $(x^i_n,\widetilde{x}^i_n)_i$ is invariant up to a permutation on the $i$'s. In particular, we have $\E\left\vert x_{n}^{i}-\widetilde{x}_{n}^{i}\right\vert=\E\left\vert x_{n}^{1}-\widetilde{x}_{n}^{1}\right\vert$, and therefore
$$ \sum_{i=1}^m \E\left\vert x_{n}^{i}-\widetilde{x}_{n}^{i}\right\vert \le m \frac{e^{(L_{b}+2L(c,\gamma ))(t-s)}(L_{b}+2L_{\mu }(c,\gamma ))(t-s)}{n}\mathcal{W}_{1}(\delta _{\rho },\mathcal{L}(\widehat{\rho }(X_{0}))).$$

\textbf{Step 4} Since the law of $X_{n}$ coincides with the law of $x_{n}$
it follows that $\mathcal{L}(\widehat{\rho }(X_{n}))=\mathcal{L}(\widehat{\rho }(x_{n}))$ and $\cL(X^1_n,\dots,X^m_n)=\cL(x^1_1,\dots,x^m_n)$.  The same is true for $\widetilde{X}_{n}$\ and $\widetilde{x}%
_{n}.$ So, we have by (\ref{chain1a}) 
\begin{eqnarray*}
\mathcal{W}_{1}\left(\mathcal{L}(\widehat{\rho }(X_{n})),\mathcal{L}(\widehat{\rho }(\widetilde{X}_{n}))\right) &=&\mathcal{W}_{1}(\mathcal{L}(\widehat{\rho }%
(x_{n})),\mathcal{L}(\widehat{\rho }(\widetilde{x}_{n})) \\
&\leq &\frac{1}{N}\sum_{i=1}^{N}\E\left\vert x_{n}^{i}-\widetilde{x}%
_{n}^{i}\right\vert \\
&\leq &\frac{e^{2(L_{b}+2L(c,\gamma ))(t-s)}(L_{b}+2L_{\mu }(c,\gamma ))(t-s)%
}{n}\mathcal{W}_{1}(\delta _{\rho},\mathcal{L}(\widehat{\rho }%
(X_{0}))).
\end{eqnarray*}
We get the other inequality by using $W_1(\cL(X^1_n,\dots,X^m_n),\cL(\widetilde{X}^1_n,\dots,\widetilde{X}^m_n))\le\E\left(\sum_{i=1}^m|x^i_n-\widetilde{x}^i_n|\right)$.
\end{proof}

\begin{proof}[Proof of Theorem~\ref{theorem_approx}]
We use the argument of Lindeberg. In order to pass
from the sequence $X_{k}$ to the sequence $\overline{X}_{k}$, we construct
 intermediary sequences as follows. Given $\kappa \in \{0,....,n-1\}$ we
 define $X_{\kappa ,k}=\overline{X}_{k}$ for $k\leq \kappa $ and, for $k\geq \kappa $ we define $X_{\kappa ,k}$ by the recurrence formula (\ref{chain5}). So the construction of $k\leq \kappa $ employs the intensity measure
based on the common law $\rho (\overline{X}_{k})$ while for $k>\kappa $ we
use the empirical measure. In particular, $X_{\kappa ,\kappa }^{i},i=1,...,N$
are independent and have the common distribution $\Theta _{s,s_{\kappa}}^{n}(\rho).$ Then we write%
\begin{equation*}
\mathcal{W}_{1}(\mathcal{L}(\widehat{\rho }(X_{n})),\mathcal{L}(\widehat{%
\rho }(\overline{X}_{n})))\leq \sum_{\kappa =0}^{n-1}\mathcal{W}_{1}(%
\mathcal{L}(\widehat{\rho }(X_{\kappa ,n})),\mathcal{L}(\widehat{\rho }%
(X_{\kappa +1,n}))).
\end{equation*}%
Let us compare the sequences $X_{\kappa ,k}$ and $X_{\kappa +1,k}$. Both
sequences start with $\overline{X}_{\kappa }$ at time $s_\kappa$ and then, in the following
step, $\rho (\overline{X}_{\kappa })$ is used to produce $X_{\kappa  ,\kappa +1}$ and the empirical measure $\widehat{\rho }(\overline{X}_{\kappa })$ is used to produce $X_{\kappa +1,\kappa +1}$. Afterwards, for $k\geq \kappa +1$
both sequences use their corresponding empirical measure. This is exactly the framework of Lemma~\ref{lem_Xtilde}, so we get 
\begin{equation*}
\mathcal{W}_{1}(\mathcal{L}(\widehat{\rho }(X_{\kappa ,n})),\mathcal{L}(%
\widehat{\rho }(X_{\kappa +1,n})))\leq \frac{e^{(L_{b}+2L_{\mu }(c,\gamma
)(t-s)}L_{\mu }(c,\gamma )(t-s)}{n}\times \mathcal{W}_1(\delta_{\Theta_{s,s_\kappa }^n(\rho)},\mathcal{L}(\widehat{\rho }(X_{\kappa}))), 
\end{equation*}%
and summing over $\kappa $ we obtain%
\begin{equation}
\mathcal{W}_{1}(\mathcal{L}(\widehat{\rho }(X_{n})),\mathcal{L}(\widehat{%
\rho }(\overline{X}_{n})))\leq \frac{e^{(L_{b}+2L_{\mu }(c,\gamma
)(t-s)}L_{\mu }(c,\gamma )(t-s)}{n}\times \sum_{\kappa=0}^{n-1} \mathcal{W}_1(\delta_{\Theta_{s,s_\kappa }^n(\rho)},\mathcal{L}(\widehat{\rho }(X_{\kappa}))).
\end{equation}%
It is well known that the moments of order~$q$ are preserved by the Euler scheme, thanks to~\eqref{growthb} and~\eqref{h6a}. We can therefore use Theorem~1 of the article~\cite{[FG1]} by Fournier and Guillin and get $\mathcal{W}_1(\delta_{\Theta_{s,s_\kappa }^n(\rho)},\mathcal{L}(\widehat{\rho }(X_{\kappa})))\le \tilde{C} M V_N$, leading to
$$ \mathcal{W}_{1}(\mathcal{L}(\widehat{\rho }(X_{n})),\mathcal{L}(\widehat{\rho }(\overline{X}_{n})))\leq  CMV_N.$$
Now using (\ref{chain1b}) with $\eta =\mathcal{L}(\widehat{\rho }(\overline{X}_{n}))$ and $\mu_{0}=\Theta _{s,s_{n}}(\rho )$ we get, for every $f$ with 
$L(f)\leq 1$%
\begin{equation*}
\E\left(\left\vert \frac{1}{N}\sum_{i=1}^{N}f(X_{n}^{i})-\int_{\R^{d}}f(x)\Theta
_{s,s_{n}}^{n}(\rho )(dx)\right\vert \right)\leq \mathcal{W}_{1}(\mathcal{L}(%
\widehat{\rho }(\overline{X}_{n})),\delta _{\Theta _{s,s_{n}}^{n}(\rho
)}))\leq CMV_N.
\end{equation*}%
Then, \eqref{chain8'} is a consequence of (\ref{W7}).

Last, the propagation of chaos follows by the same arguments, since we have from Lemma~\ref{lem_Xtilde}
\begin{align*}
 & W_1\left(\cL(X_{\kappa,n}^1,\dots,X_{\kappa,n}^m),\cL(X_{\kappa+1,n}^1,\dots,X_{\kappa+1,n}^m)\right)\\
  &\le m \frac{e^{(L_{b}+2L_{\mu }(c,\gamma
)(t-s)}L_{\mu }(c,\gamma )(t-s)}{n}\times \mathcal{W}_1(\delta_{\Theta_{s,s_\kappa }^n(\rho)},\mathcal{L}(\widehat{\rho }(X_{\kappa}))).
\end{align*}
\end{proof}

\textbf{Approximating particles system and algorithm}

We now discuss briefly the problem of sampling the system of particles defined
by~\eqref{chain5}. To do so, we will assume that: 
\begin{equation}
\mu (E)<\infty \quad \text{ and }\quad \left\vert \gamma (v,z,x)\right\vert
\leq \Gamma ,\forall v,x\in {\mathbb{R}}^{d},z\in E.  \label{chain12a}
\end{equation}%
The approximation in a more general framework requires then to use some
truncation procedures and to quantify the corresponding error.

When~\eqref{chain12a} holds, the solution of~\eqref{chain5} is constructed
in an explicit way as follows. Let us assume that the values of $(X_{k}^{i},i\in \{1,\dots ,N\})$ have been obtained: we explain how to construct then $(X_{k+1}^{i},i\in \{1,\dots ,N\})$. We take $T_{\ell}^{i},\ell \in {\mathbb{N}}$ to be the jump times of a Poisson process of
intensity $\mu (E)\times \Gamma $ and we take $Z_{\ell }^{i}\sim \frac{1}{%
\mu (E)}\mu (dz),U_{\ell }^{i}\sim \frac{1}{\Gamma }1_{[0,\Gamma ]}(u)du$
and $\varepsilon _{\ell }^{i}$ uniformly distributed on $\{1,...,N\}$. For
each $i=1,...,N$ this set of random variables are independent. Then one
computes explicitly 
\begin{equation*}
X_{k+1}^{i}=X_{k}^{i}+b(X_{k}^{i},\widehat{\rho }(X_{k}))\frac{s-t}{n}%
+\sum_{s_{k}\leq T_{\ell }^{i}<s_{k+1}}Q(X_{k}^{\varepsilon _{\ell
}^{i}},Z_{\ell }^{i},U_{\ell }^{i},X_{k}^{i},\widehat{\rho }(X_{k})),
\end{equation*}%
which gives the desired particle system that satisfies~\eqref{chain8'}.

\bigskip


\section{The Boltzmann equation}\label{sec_Boltz}

\subsection{The homogeneous Boltzmann equation}

We consider the following more specific set of coefficients, which
corresponds to the Boltzmann equation with hard potential. We take $a\in (0,1)$ and we define
\begin{equation}\label{gamma_Boltz}
  \gamma(v,x)=\left\vert v-x\right\vert ^{a}.
\end{equation}
Moreover we take  $b:{\mathbb{R}}^{d}\rightarrow {\mathbb{R}}^{d}$ that  is Lipschitz continuous (and thus satisfies (\ref{lipb}))
and $c:{\mathbb{R}}^{d}\times E\times {\mathbb{R}}^{d}\rightarrow {\mathbb{R}}^{d}$ that verifies the following
hypothesis. We assume that for every $(v,x)\in {\mathbb{R}}^{d}\times {\mathbb{R}}^{d}$ there exists a function $c_{v,x}:{\mathbb{R}}^{d}\times E \times {\mathbb{R}}^{d}\rightarrow \R$ such that for every $v',x'\in {\mathbb{R}}^{d}$ and $\varphi \in C_{b}^{1}({\mathbb{R}}^{d})$%
\begin{equation}
\int_{E}\varphi (c(v',z,x')) \mu (dz)=\int_{E}\varphi
(c_{v,x}(v',z,x'))\mu (dz).  \label{H1}
\end{equation}%
Notice that, since $\gamma $ does not depend on $z$, this guarantees that $Q(v,z,u,x):=c(v,z,x)1_{\{u<\gamma (v,x)\}}$ verifies (\ref{h6e}). Then, we
assume that there exists some function $\alpha :E\rightarrow {\mathbb{R}}_{+}$ such that $\int_{E}\alpha (z)\mu (dz)<\infty $,
\begin{equation}
\left\vert c(v,z,x)\right\vert \leq \alpha (z)\left\vert v-x\right\vert
\quad\text{and}\quad \left\vert c(v,z,x)-c_{v,x}(v',z,x')\right\vert \leq \alpha
(z)(\left\vert v-v'\right\vert +\left\vert x-x'\right\vert ).  \label{H2}
\end{equation}%
Notice that $Q$ may not satisfy the sublinear growth property~(\ref{h6a}) because~\eqref{H2} only ensures that  $\int_{E\times \R_+}|Q(v,z,u,x)| \mu(dz)du =  \gamma(v,x) \int_{E}|c(v,z,x)| \mu(dz)\le\int_E \alpha(z)\mu(dz) |v-x|^{1+a}$. Thus~\eqref{h6d} may not hold. So, our results does not apply directly for these coefficients. In order to fit our framework, we will use a truncation procedure. For $\Gamma \geq 1$ we define $H_{\Gamma }(v)=v\times \frac{\left\vert v\right\vert \wedge \Gamma }{\left\vert v\right\vert }$ and we notice that $\left\vert H_{\Gamma}(v)\right\vert \leq \Gamma $ and $\left\vert H_{\Gamma }(v)-H_{\Gamma}(w)\right\vert \leq \left\vert v-w\right\vert$. Then we define
\begin{align}
&\gamma _{\Gamma }(v,x) =\gamma (H_{\Gamma }(v),H_{\Gamma }(x))=\left\vert
H_{\Gamma }(v)-H_{\Gamma }(x)\right\vert ^{a},  \notag\\
&c_{\Gamma }(v,z,x) =c(H_{\Gamma }(v),z,H_{\Gamma }(x)), \quad c_{\Gamma,(v,x)}(v',z,x')=c_{H_\Gamma(v),H_\Gamma(x)}(H_\Gamma(v'),z,H_\Gamma(x')),\label{H2a}\\
&Q_{\Gamma}(v,z,u,x)=Q(H_{\Gamma }(v),z,u,H_{\Gamma }(x)),  \quad Q_{\Gamma,(v,x)}(v',z,u,x')=1_{u<\gamma_\Gamma(v',x')}c_{\Gamma,(v,x)}(v',z,x'). \notag
\end{align}%

\begin{lemma}
  \label{TR} Let $b:\R^d\to \R^d$ be Lipschitz continuous and assume~\eqref{H1} and~\eqref{H2}. Then, the triplet $(b,c,\gamma)$ satisfies~\eqref{lipb} and \eqref{h6e}.

  Besides, for every $\Gamma \geq 1,$ the triplet $(b,c_{\Gamma },\gamma_{\Gamma })$ satisfies~$\mathbf{(A)}$  with   $L_{\mu}(c_\Gamma,\gamma_\Gamma )=6  \Gamma ^{a}\int_E \alpha(z) dz$. 
\end{lemma}
The proof of this Lemma is postponed to Appendix~\ref{App_proofs}. Thanks to this result, we can then apply Theorem~\ref{flow} to construct a flow $\theta _{s,t}^{\Gamma }(\rho)$. By Theorem~\ref{Weq}, this flows solves the weak equation~\eqref{we2} associated with $\gamma_\Gamma$ and $Q_\Gamma$. Besides, by Theorem~\ref{ExistenceRepr} there exists a probabilistic representation of this solution.  The natural question is then to know if $\theta _{s,t}^{\Gamma }(\rho)$ converges  when $\Gamma \rightarrow \infty $. This would produce a flow that would be a natural candidate for the solution of the Boltzmann equation. We leave this issue for further research.

\subsubsection*{The 3D Boltzmann equation with hard potential} \label{subsec_HP}

We now precise the coefficients which appear in the homogeneous Boltzmann
equation in dimension three. We follow the parametrization introduced in~\cite{[FM]} and~\cite{[F1]}. For this equation, the space $E$ is $E=[0,\pi ]\times [ 0,2\pi ]$, we note $z=(\zeta ,\varphi )$ and the measure $\mu$ is defined by $\mu(dz)=\zeta ^{-(1+\nu )}d\zeta d\varphi$, for some  $\nu \in (0,1)$. The coefficient $\gamma$ is given by~\eqref{gamma_Boltz}. We now define~$c$. Given a vector $X\in {\mathbb{R}}^{3}\setminus\{0\}$, one
may construct $I(X),J(X)\in {\mathbb{R}}^{3}$ such that $X\rightarrow(I(X),J(X))$ is measurable and $(\frac{X}{\left\vert X\right\vert },\frac{I(X)}{\left\vert X\right\vert },\frac{J(X)}{\left\vert X\right\vert })$ is an orthonormal basis in ${\mathbb{R}}^{3}$. We define the function $\Delta
(X,\varphi )=\cos (\varphi )I(X)+\sin (\varphi )J(X)$  and then
\begin{equation}\label{def_c_3D}
c(v,(\zeta ,\varphi ),x)=-\frac{1-\cos \zeta }{2}(v-x)+\frac{\sin (\zeta )}{2}\Delta (v-x,\varphi ).
\end{equation}%
The specific difficulty in this framework is that $c$ does not satisfy the
standard Lipschitz continuity property. It has been circumvented by Tanaka in \cite{[T1]} (see also Lemma 2.6 in \cite{[FM]}) who proves that one may construct a measurable function $\eta :{\mathbb{R}}^{3}\times {\mathbb{R}}^{3}\rightarrow \lbrack 0,2\pi ]$ such that 
\begin{equation*}
\left\vert c(v,(\zeta ,\varphi ),x)-c(v',(\zeta
,\varphi +\eta (v'-x',v-x)),x')\right\vert \leq 2\zeta (\left\vert v-v'\right\vert +\left\vert x-x'\right\vert).
\end{equation*}%
This means that Hypothesis~(\ref{H2}) holds with $c_{v,x}(v',(\zeta ,\varphi) ,x')=c(v',(\zeta ,\varphi +\eta (v'-x',v-x)),x')$. This function also satisfies (\ref{H1}):
\begin{equation*}
\int_{0}^{\pi }\frac{d\zeta }{\zeta ^{1+\nu }}\int_{0}^{2\pi}f(x+c(v,(\zeta ,\varphi ),x))d\varphi =\int_{0}^{\pi }\frac{d\zeta }{\zeta ^{1+\nu }}\int_{0}^{2\pi }f(x+c(v,(\zeta ,\varphi +\eta (v-x,\overline{v}-\overline{x})),x))d\varphi,
\end{equation*}%
since for every $v,x\in {\mathbb{R}}^{3}$ and $\zeta \in (0,\pi )$ the
function $\varphi \rightarrow f(x+c(v,(\zeta ,\varphi ),x))$ is $2\pi$-periodic. 
We are therefore indeed in the framework of Lemma~\ref{TR}.

\subsection{The  Boltzmann-Enskog equation}

In this section we consider the non homogeneous Boltzmann equation called Enskog equation which has been discussed in~\cite{[Ar]}. The study of this equation has been initiated in~\cite{[P]}, and more recent contributions concerning existence, uniqueness, probabilistic interpretation and particle system approximations are given in~\cite{[ARS]}, \cite{[FRS]} and~\cite{[FRS1]}. We consider a model in which $\bX=(\bX^{1},...,\bX^{d})\in {\mathbb{R}}^{d}$ with $d=3$
represents the position of the "typical particle" and $X=(X^{1},...,X^{d})\in{\mathbb{R}}^{d}$ is its velocity. In all this subsection, letters with bar will refer to positions, and bold letters $\boX=(\bX,X)$ will denote the couple position-velocity. Then, the position
follows the dynamics given by the velocity:%
\begin{equation*}
  \bX_{s,t}=\bX_{0}+\int_{s}^{t}X_{s,r}dr,
\end{equation*}%
where $\bX_0$ is an integrable random variable. As for the velocity, $X_{s,t}$ follows the equation
\begin{equation*}
X_{s,t}=X_{0}+\int_{s}^{t}\int_{({\mathbb{R}}^{d}\times {\mathbb{R}}^{d})\times E\times {\mathbb{R}}_{+}} c(v,z, X_{s,r-})1_{\{u\leq \gamma(v,X_{s,r-})\}}\times \beta (\bv,\bX_{s,r-}) N(d\bov,dz,du,dr)
\end{equation*}%
where $\bov=(\bv,v)$, $\gamma(v,x)=|x-v|^a$, $\beta \in C_{b}^{1}({\mathbb{R}}^{d}\times {\mathbb{R}}^{d})$ and $N$ is a Poisson point measure of intensity 
\begin{equation*}
\botheta_{s,r}(d\bov )\mu (dz)dudr\quad \text{with}\quad
\botheta_{s,r}(d \bov)=\P((\bX_{s,r},X_{s,r})\in d\bov).
\end{equation*}%
Here, as in the case of the homogeneous equation, $E=[0,\pi ]\times [0,2\pi ]$ and $z=(\zeta ,\varphi )$ and the measure $\mu (dz)=\zeta ^{-(1+a)} d\zeta d\varphi$, $a\in (0,1)$.

We look to this dynamics as a system in dimension $2d$ (typically with $d=3$). We denote $\boxx=(\bx,x)$, $\bov=(\bv,v)$ and $\boX=(\bX,X)$. The drift is then given by 
\begin{eqnarray} 
\bob(\boxx) &=&\begin{cases} x^{i}\quad \text{for }i=1,...,d,  \\
0\quad \text{for }i=d+1,...,2d,
\end{cases}\label{e1}
\end{eqnarray}%
and the collision kernel (cross section) is%
\begin{equation}
\boc (\bov,z,\boxx)=c (v,z,x) \beta (\bv,\bx)  \label{e2},
\end{equation}%
where~$c$ is defined as in the Boltzmann equation by~\eqref{def_c_3D} with $z=(\zeta ,\varphi )\in E$, and
$$ \bogamma (\boxx,\bov)=\left\vert v-x\right\vert^{a},$$
with $a\in (0,1)$. The equation~(\ref{we2}) associated to these coefficients is the ($d$-dimensional) Enskog equation. In the particular case $\beta(\bv,\bx)=1$, we recover the case of the homogeneous Boltzmann equation and $\bar{X}$ is just the time-integral of the process. The specificity of the inhomogeneous case is illustrated by the following example. Let us take $R>0$,  $i_{R}$ be a regularized version of the
indicator function $1_{x<R}$ and define  $\beta_{R}(\bv,\bx)=i_{R}(\left\vert \bx-\bv \right\vert )$. Then, the coefficient $\beta_{R}(\bv,\bx)$  means that only the particles which are closer to the distance $R$ may collide.

We now define the truncated coefficients $\boc_\Gamma$ and $\bogamma_\Gamma$ for $\Gamma>0$. We still denote, for $v\in\R^d$, $H_{\Gamma }(v)=v\times \frac{\left\vert v\right\vert \wedge \Gamma }{\left\vert v\right\vert }$, and we define  $\boc_\Gamma (\bov,z,\boxx)=c_\Gamma(v,z,x) \beta_\Gamma(\bv,\bx)$ with  $\beta_\Gamma(\bv,\bx):=\beta(H_\Gamma(\bv),H_\Gamma(\bx))$ and  $\bogamma_\Gamma (\boxx,\bov)=\left\vert H_\Gamma(v)-H_\Gamma(x)\right\vert^{a}$.

\begin{lemma}\label{TR2} Assume that $\beta \in C_{b}^{1}({\mathbb{R}}^{d}\times {\mathbb{R}}^{d})$  and that~\eqref{H1} and~\eqref{H2} hold. Then,  for every $\Gamma \ge 1$, the triplet $(\bob,\boc_\Gamma,\bogamma_\Gamma)$ satisfies~$\mathbf{(A)}$  with   $L_{\mu}(c_\Gamma,\gamma_\Gamma )=C  \Gamma ^{a+1}$ for some constant $C>0$.\end{lemma}

The proof is postponed to Appendix~\ref{App_proofs}. As for the Boltzmann equation, this lemma allows by Theorem~\ref{flow} to construct the flow, and then by Theorems~\ref{Weq} and~\ref{ExistenceRepr}, a weak solution and a probabilistic representation for the Enskog-Boltzmann equation with truncated coefficients. The convergence when $\Gamma \to \infty$ remains an open problem.

\subsection{A Boltzmann equation with a mean field interaction on the position}

The fact that we are able with our approach to mix easily Boltzmann and McKean-Vlasov interactions gives more flexibility to model the behaviour of particles. Here,
we give a very simple example that is derived from the Boltzmann-Enskog
equation discussed above. In this equation, interactions are made both for the position  and the velocity through a Poisson point measure. More precisely, when the function $\beta$ is the indicator function, collisions only occur when the particle is close enough to the one sampled by the Poisson point measure.  Here, we consider an alternative modelling with a mean field interaction on the position.  Precisely, we consider the coefficient~$\bob$ defined by~\eqref{e1}, and the coefficients $\boc$ and $\bogamma$ given by
\begin{equation}\label{BE_meanfield}
  \boc(v,z,\boxx)=c(v,z,x) \text{ and } \bogamma(v,\boxx,\rho)=|v-x|^a\int_{\R^d}p_{R}(\bx-\bx')\rho (d \bx'),
\end{equation}
with $p_R(x)=\frac 1 {(2\pi R^2)^{d/2}}\exp\left(-\frac{|x|^2}{2R^2}\right)$ is the Gaussian kernel. Thus, the higher is the probability density function of the particles in the neighbourhood of~$\bx$, the more likely are the collisions. This corresponds to the following dynamics for $\boX=(\bX,X)$:
\begin{equation}\label{BE_MFeq}
 \begin{cases} \bX_{s,t}=\bX_{0}+\int_{s}^{t}X_{s,r}dr,\\
X_{s,t}=X_{0}+\int_{s}^{t}\int_{ {\mathbb{R}}^{d}\times E\times {\mathbb{R}}_{+}} c(v,z, X_{s,r-})1_{\{u\leq \bogamma(v,\boX_{s,r-},\bar{\theta}_{s,r})\}} N(dv,dz,du,dr),
 \end{cases}
\end{equation}
where $N$ is a Poisson point measure with intensity $\theta_{s,r}(dv)\mu(dz)dudr$. Here,  $\botheta_{s,r}(d \bov)=\P((\bX_{s,r},X_{s,r})\in d\bov)$ is the probability distribution of $(\bX_{s,r},X_{s,r})$ with marginal laws $\bar{\theta}_{s,r}(d\bv)=\P(\bX_{s,r}\in d\bv)$ and $\theta_{s,r}(dv)=\P(X_{s,r}\in dv)$.

To underline the different types of interaction (Boltzmann and mean-field), we also write the interacting particle system corresponding to this equation: for $i\in\{1,\dots,N\}$,
\begin{equation}\label{BE_MFeq_IPS}
 \begin{cases} \bX^i_{t}=\bX^i_{0}+\int_{0}^{t}X^i_{r}dr,\\
X^i_{t}=X^i_{0}+\int_{0}^{t}\int_{ {\mathbb{R}}^{d}\times E\times {\mathbb{R}}_{+}} c(v,z, X^i_{s,r-})1_{\{u\leq   \frac 1N \sum_{j=1}^N |X^i_{t-}-X^j_{t-}|^a p_R(\bX^i_{t}-\bX^j_{t})\}} N^i(dv,dz,du,dr),
 \end{cases}
\end{equation}
where $N^i$, $i=1,\dots,N$, are independent Poisson point measure with intensity $\frac 1 N \sum_{j=1}^N\delta_{X^j_{t-}}(dv)\mu(dz) du dr$.

We now define for $\Gamma>0$ the truncated coefficients $\boc_\Gamma(v,z,\boxx)=c(H_\Gamma(v),z,H_\Gamma(x))$ and $\bogamma_\Gamma (v,\boxx,\rho)=|H_\Gamma(v)-H_\Gamma(x)|^a\int_{\R^d}p_{R}(\bx-\bx')\rho (d \bx')$. The proof of next lemma can be found in Appendix~\ref{App_proofs}.
\begin{lemma}\label{TR3}
  Assume that~\eqref{H1} and~\eqref{H2} hold. Then, for every $\Gamma \ge 1$, the triplet $(\bob_\Gamma,\boc_\Gamma,\bogamma_\Gamma)$ defined as the truncation of~\eqref{e1} and~\eqref{BE_meanfield} satisfies~$\mathbf{(A)}$  with   $L_{\mu}(c_\Gamma,\gamma_\Gamma )=C  \Gamma ^{a+1}$ for some constant $C>0$.
\end{lemma}

Thanks to Lemma~\ref{TR3}, we can apply Theorem~\ref{flow} to get the existence and uniqueness of a flow~$\theta_{s,t}^\Gamma$ corresponding to this equation and Theorem~\ref{ExistenceRepr} to get a probabilistic representation associated to this flow. Again, one needs further assumptions to justify that this produces, when $\Gamma \to \infty$ a solution to~\eqref{BE_MFeq}. This is left for further research. 

\appendix

\section{Technical proofs}\label{App_proofs}

\begin{proof}[Proof of Lemma~\ref{TR}]
  It is easy to check that $c_\Gamma$ also satisfies~\eqref{H1} and~\eqref{H2}.  Therefore~\eqref{h6e} holds for $(b,c_{\Gamma },\gamma_{\Gamma })$. We now check~(\ref{h6d}). We have%
\begin{align*}
&\int_{E\times {\mathbb{R}}_{+}}\left\vert Q_{\Gamma
}(v_{1},z,u,x_{1})-Q_{\Gamma ,(v_{1},x_{1})}(v_{2},z,u,x_{2})\right\vert
  \mu (dz)du \le A+B, \text{ with}\\
A= &\int_{E\times {\mathbb{R}}_{+}}\left\vert c_{\Gamma }(v_{1},z,x_{1})-c_{\Gamma
    ,(v_{1},x_{1})}(v_{2},z,x_{2})\right\vert 1_{u<\gamma_{\Gamma }(v_{2},x_{2})}\mu (dz)du,\\
  B= &\int_{E\times {\mathbb{R}}_{+}} \left\vert c_{\Gamma
}(v_{1},z,x_{1}) ( 1_{u<\gamma_{\Gamma }(v_{1},x_{1})}- 1_{u<\gamma_{\Gamma }(v_{2},x_{2})})\right\vert
 \mu (dz)du.
\end{align*}%
From~\eqref{H2} and $|\gamma_\Gamma(v,x)|\le 2\Gamma^{a}$, we get $A\le 2\Gamma^a \int_E \alpha(z)\mu(dz)(|v_1-v_2|+|x_1-x_2|)$. We have 
\begin{equation}
B\le \int_{E}\left\vert \gamma _{\Gamma }(v_{1},x_{1})-\gamma _{\Gamma
}(v_{2},x_{2})\right\vert |c_{\Gamma }(v_{1},z,x_{1})|\mu (dz),  \label{b1}
\end{equation}%
and  we denote%
\begin{equation*}
X=\left\vert H_{\Gamma }(v_{1})-H_{\Gamma }(x_{1})\right\vert \quad \text{and} \quad
Y=\left\vert H_{\Gamma }(v_{2})-H_{\Gamma }(x_{2})\right\vert
\end{equation*}%
We now use the following basic inequality:
$$\forall x,y>0, (x+y)|x^a-y^a|\le (x^a+y^a)|x-y|.$$
This is easily checked by taking for example $x>y$, expanding and observing that $yx^a-xy^a=(xy)^a(y^{1-a}-x^{1-a})<0$ since $a\in (0,1)$. We then get 
\begin{eqnarray*}
X(X^{a}-Y^{a}) &\leq &(X+Y)(X^{a}-Y^{a})\leq (X^{a}+Y^{a})\left\vert X-Y\right\vert \leq 
4\Gamma ^{a}\left\vert X-Y\right\vert \\
&\leq &4\Gamma ^{a}(\left\vert H_{\Gamma
}(v_{1})-H_{\Gamma }(v_{2})\right\vert +\left\vert H_{\Gamma
}(x_{1})-H_{\Gamma }(x_{2})\right\vert ) \\
&\leq &4 \Gamma ^{a} (\left\vert v_{1}-v_{2}\right\vert
+\left\vert x_{1}-x_{2}\right\vert ).
\end{eqnarray*}%
We also have 
\begin{equation*}
\left\vert c_{\Gamma }(v_{1},z,x_{1})\right\vert \leq \alpha (z)\left\vert
H_{\Gamma }(v_{1})-H_{\Gamma }(x_{1})\right\vert =\alpha(z) X
\end{equation*}%
so we get 
\begin{equation} \label{b2}
B\leq X(X^{a}-Y^{a})\times \int_{E}\alpha(z)\mu (dz)\leq  4 \Gamma ^{a} \int_{E}\alpha(z)\mu (dz) (\left\vert v_{1}-v_{2}\right\vert +\left\vert x_{1}-x_{2}\right\vert ).
\end{equation}
\end{proof}
\begin{proof}[Proof of Lemma~\ref{TR2}]
  Using that $\beta$ is bounded, we have
  $|\boc_\Gamma(\bov,z,\boxx)|\le \|\beta\|_\infty \alpha(z)|v-x|$. Besides, we have
  \begin{align*}
    & |c_\Gamma(v,z,x)\beta_\Gamma(\bv,\bx)-c_{\Gamma,(v,x)}(v',z,x')\beta_\Gamma(\bv',\bx')|\\
    &\le |c_\Gamma(v,z,x)||\beta_\Gamma(\bv,\bx)-\beta_\Gamma(\bv',\bx')|+|\beta_\Gamma(\bv',\bx')||c_\Gamma(v,z,x)-c_{v,x}(v',z,x')| \\
    &\le 2\alpha(z)\Gamma L(\beta)(|\bv-\bv'|+|\bx-\bx'|)+\alpha(z)\|\beta\|_\infty(|v-v'|+|x-x'|),
  \end{align*}
  where $L(\beta)$ is the Lipschitz constant of~$\beta$.
Thus, if we define $\boc_{\bov,\boxx}(\bov',z,\boxx')=c_{v,x}(v',z,x')\beta(\bv',\bx')$, $\boc$ satisfies equations~\eqref{H1} and~\eqref{H2} with $\boldsymbol{\alpha}(z)=(2\Gamma L(\beta)+\|\beta\|_\infty)\alpha(z)$. We then get the claim by applying the same arguments as for Lemma~\ref{TR}.   
\end{proof}
\begin{proof}[Proof of Lemma~\ref{TR3}]
  We first observe that $\|p_R\|_\infty=(2\pi R^2)^{-d/2}$ and that $p_R$ is Lipschitz with constant $\|p_R\|_\infty /R$. Thus, the proof of Lemma~\ref{TR} can be easily adapted. In particular, the term ``B'' is now
\begin{align*}
  B&=\int_{E\times \R_+}|c_\Gamma(v_1,z,x_1)1_{u\le \gamma_\Gamma(v_1,\boxx_1,\rho_1)}-c_\Gamma(v_2,z,x_2)1_{u\le \gamma_\Gamma(v_2,\boxx_2,\rho_2)}|\mu(dz)du    \\
  &\le \int_{E\times \R_+}|c_\Gamma(v_1,z,x_1)||\gamma_\Gamma(v_1,\boxx_1,\rho_1)-\gamma_\Gamma(v_2,(\bx_1,x_2),\rho_1)| \mu(dz)\\
&+ \int_{E\times \R_+}|c_\Gamma(v_1,z,x_1)||\gamma_\Gamma(v_2,(\bx_1,x_2),\rho_1)-\gamma_\Gamma(v_2,\boxx_2,\rho_2)| \mu(dz).
\end{align*}
The first term can be bounded by~$4\|p_R\|_\infty\Gamma^a\int_E\alpha(z)dz(|v_1-v_2|+|x_1-x_2|)$ as in Lemma~\eqref{TR} while the second term is bounded by
\begin{align*}
  &2\Gamma \times 2\Gamma^a \times\left|\int_{\R^d}p_{R}(\bx_1-\bx')\rho_1 (d \bx') -\int_{\R^d}p_{R}(\bx_2-\bx')\rho_2 (d \bx') \right|\\
  &\le \frac{4\Gamma^{a+1}\|p_R\|_\infty}{R}\left(|\bx_1-\bx_2|+W_1(\rho_1,\rho_2)\right).
\end{align*}
The other properties are easy to check.
\end{proof}
\bibliographystyle{abbrv}
\bibliography{biblio_Boltzmann}
\end{document}